\documentclass[11pt]{amsart}
\usepackage{amsmath,amsfonts,amssymb,amsthm}
\usepackage{graphics}
\usepackage{graphicx}
\usepackage{overpic}
\usepackage{epsfig,psfrag}
\usepackage{color}

\numberwithin{equation}{section}

\newtheoremstyle{thm}
  {9pt}{9pt}{\itshape}{}{\bfseries}{}{.5em}{}

\theoremstyle{thm}
\newtheorem{thm}{Theorem}[section]
\newtheorem{cor}[thm]{Corollary}
\newtheorem{lemma}[thm]{Lemma}

\newtheoremstyle{defin}
  {9pt}{9pt}{}{}{\bfseries}{}{.5em}{}
\theoremstyle{defin}

\newtheorem{cl}{Claim}
\newtheoremstyle{exm}
  {9pt}{9pt}{}{}{\scshape}{}{.5em}{}
\theoremstyle{exm}

\newtheoremstyle{proof}
  {}{}{}{}{\itshape}{:}{.5em}{}
\theoremstyle{proof}

\newcommand{\ndash}{\nobreakdash--}

\newcommand{\Z}{{\mathbb Z}}
\newcommand{\R}{{\mathbb R}}

\DeclareMathOperator{\conv}{Conv}

\def\F{\mathbf F}

\def\rr{\mathbb R}
\def\qqq{\mathbb Q}

\def\sm{\smallsetminus}

\def\cF{\mathcal F}

\def\cP{\mathcal P}

\def\<{\langle}
\def\>{\rangle}

\def\0{{\mathbf 0}}

\def\.{\hskip.06cm}
\def\ts{\hskip.03cm}

\def\conv{{\text {\rm {conv}} }}

\def\ba{\textbf{a}}
\def\bx{\textbf{x}}
\def\by{\textbf{y}}
\def\bz{\textbf{z}}

\newcommand{\A}{\mathcal{A}}

\renewcommand{\P}{\mathcal{P}}

\newcommand{\Rc}{\mathcal{R}}
\newcommand{\Sc}{\mathcal{S}}

\title[A Quantitative Steinitz Theorem]{A Quantitative Steinitz Theorem for Plane Triangulations}
\author[Igor Pak and Stedman Wilson]{Igor Pak$^\ast$ \, and \, Stedman Wilson$^\dagger$}

\date{\today}

\thanks{\thinspace ${\hspace{-.45ex}}^\ast$Department of Mathematics, UCLA, Los Angeles, CA 90095, USA; \ts
\texttt{\{pak\}@math.ucla.edu}}

\thanks{\thinspace ${\hspace{-.45ex}}^\dagger$Department of Mathematics, Ben Gurion University,
Be'er Sheva, Israel; \ts \texttt{\{stedman\}@math.bgu.ac.il}}

\begin{document}

\begin{abstract}
We give a new proof of Steinitz's classical theorem in the case of plane triangulations,
which allows us to obtain a new general bound on the grid size of the simplicial polytope
realizing a given triangulation, subexponential in a number of special cases.

Formally, we prove that every plane triangulation $G$ with $n$ vertices can be
embedded in~$\rr^2$ in such a way that it is the vertical projection of a convex polyhedral
surface.  We show that the vertices of this surface may be placed in a
$4n^3 \times 8n^5 \times \zeta(n)$ integer grid, where $\zeta(n) \leq (500 \ts n^8)^{\tau(G)}$ and
$\tau(G)$ denotes the \emph{shedding diameter} of $G$, a quantity defined in the paper.
\end{abstract}

\maketitle

\section{Introduction}


\noindent
The celebrated \emph{Steinitz's theorem} states every $3$-connected plane graph~$G$
is the graph of a $3$-dimensional convex polytope.  An important corollary of the original proof
is that the vertices of the polytope can be made integers.  The
\emph{Quantitative Steinitz Problem}~\cite{R} asks for the smallest size of such
integers as they depend on a graph.  The best current bounds are exponential in the number
of vertices in all three dimensions, even when restricted to triangulations, see~\cite{RRS}.
In this paper we improve these bounds in two directions.  While the main result of
this paper is rather technical (Theorem~\ref{t:size}), the following corollary
requires no background.

\begin{cor}
Let $G$ be a plane triangulation with $n$ vertices.  Then~$G$
is a graph of a convex polyhedron with vertices lying in a \.
$4\ts n^3 \times 8\ts n^5 \times (500\ts n^8)^{n}$ \. integer grid.
\label{c:size}
\end{cor}

This result improves known bounds in two directions at the expense of a somewhat
weaker bound in the third direction.  However, for large families of graphs we
make sharp improvements in the third direction as well.  Below we give our
our main application.

A \emph{grid triangulation} of
\ts $[a \times b] = \{1, \ldots, a\} \times \{1, \ldots, b\}$ \ts
is a triangulation with all grid points as the set of vertices.
These triangulations have a curious structure, and have been studied
and enumerated in a number of papers (see~\cite{A,KZ,W} and references
therein).

\begin{cor}
Let $G$ be a grid triangulation of $[k \times  k]$, such that every triangle
fits in an $\ell \times \ell$ subgrid.  Then $G$ is a graph
of a convex polyhedron with vertices lying in a \.
$O(k^6) \times O(k^{10}) \times k^{O(\ell k)}$ \. integer grid.
\label{c:grid}
\end{cor}

Setting $\ell=O(1)$ as $k\to \infty$, for the grid triangulations as in the
corollary, we have a subexponential grid size in the number $n=k^2$ of vertices:  \ts
$O(n^3)\ts \times\ts O(n^5)\ts \times \. \exp O(\sqrt{n} \ts \log n)$.

\begin{figure}[ht!]
 \begin{center}
   \includegraphics[height=3.5cm]{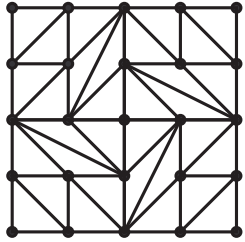}
   \caption{An example of a grid triangulation of $[5 \times 5]$, with $\ell = 3$.}
   \label{f:santos}
 \end{center}
\end{figure}

\medskip

The basic idea behind the best known bounds in the quantitative
Steinitz problem, is as follows (see~\cite{R,RRS,Ro}).
Start with the \emph{Tutte spring embedding} of~$G$ with unit weights~\cite{T},
and lift it up to a convex surface according to the \emph{Maxwell--Cremona theorem}
(see~\cite{L,R}).  Since Tutte's embedding and the lifting are given by
rational equations, this embedding can be expanded to an integer embedding.  However,
there is only so much room for this method to work, and since the determinants are given
by the number of spanning trees in~$G$, the bounds cannot be made subexponential in the
case of triangulations.

Although there are several interesting proofs of the Steinitz theorem~\cite{Z1,Z2},
neither seem to simplify in the case of triangulations.  The proof we present follows
a similar idea, but in place of the Tutte spring embedding we present an inductive
construction.  In essence, we construct a \emph{strongly convex embedding} of
plane triangulations, based on a standard inductive proof of \emph{F\'ary's theorem}~\cite{F}.
We make our construction quantitative, by doing this on a $O(n^3) \times O(n^5)$ grid,
thus reproving a weak version of the main result in~\cite{BR}.

We then lift the resulting triangulation directly to a convex surface.  The inductive
argument allows us to obtain a new type of quantitative bound $\zeta(n)= n^{O(\tau(G))}$
on the height of the lifting.  The parameter $\tau(G)$ here may be linear in~$n$
in the worst case.  It is bounded from below by both the diameter of~$G$ and the diameter
of the dual graph of~$G$.  However, this parameter is sublinear in a number of special cases,
such as the grid triangulations mentioned above (see~$\S$\ref{ss:fin-diam}).

\smallskip

The rest of this paper is structured as follows.  In the next section we recall some
definitions and basic results on \emph{graph drawing}.  In Section~\ref{s:grid} we
prove Theorem~\ref{plane}, the crucial technical result on graph embedding.
Then, in Section~\ref{s:diam}, we define the \emph{shedding diameter} and prove
Theorem~\ref{t:size}, the main result of this paper.  We discuss grid triangulations
in Section~\ref{s:grid-tri}, and conclude with final remarks in Section~\ref{s:fin}.

\bigskip

\section{Definitions and basic results} \label{s:def}


\noindent
Let $G = (V, E)$ denote a plane graph.  By abuse of notation we will identify $G$ with the subset of $\R^2$ consisting of its vertices and
edges.  We write $V(G)$ for the vertices of $G$ and $E(G)$ for the edges of $G$.  When $G$ is $2$-connected we let $\cF(G) = \{F_1, \ldots, F_m\}$ denote the set
of (closed) bounded faces of $G$.  We define $\F(G) = \bigcup_i F_i$, the region of $\R^2$ determined by $G$.  For a subgraph $H$ of $G$, we write $H \subseteq
G$.

When $G$ is $2$-connected, a vertex $v \in V$ is called a \textit{boundary vertex} if $v$ is in the boundary of $\F(G)$, and an \emph{interior} vertex otherwise.
Similarly, an edge $e \in E$ is called a \emph{boundary edge} if $e$ is completely contained in the boundary of $\F(G)$, and an \emph{interior edge} otherwise.
A \emph{diagonal} of $G$ is an interior edge whose endpoints are boundary vertices of $G$.  For a plane graph $G$ with vertex $v$, let $G - \{v\}$ denote the plane graph obtained by removing $v$ and all edges adjacent to $v$.  Let $\Delta(G)$ denote the diameter of $G$.

We say that two plane graphs $G, G'$ are \textit{face isomorphic}, written $G \sim G'$, if there is a graph isomorphism $\psi : V(G) \rightarrow V(G')$ that also induces a
bijection $\psi_\cF : \cF(G) \rightarrow \cF(G')$ of the bounded faces of $G$ and $G'$.  This last property means that $v_1, \ldots, v_k$ are the vertices of a face $F \in \cF(G)$ if and only if $\psi(v_1), \ldots,
\psi(v_k)$ are the vertices of a face $F' \in \cF(G')$.  By definition, $G \sim G'$ implies that $G$ and $G'$ are isomorphic as abstract graphs, but the converse
in not always true.  When $G \sim G'$ and $v$ is a vertex of~$G$, we will write $v'$ for the corresponding vertex of~$G'$, indicating that a face isomorphism $\psi$ is defined by $v' = \psi(v)$.

A \emph{geometric} plane graph is a plane graph for which each edge is a straight line segment.  A \emph{geometric embedding} of a plane graph $G$ \emph{in the
set} $S \subseteq \R^2$ is a geometric plane graph $G'$ such that $G \sim G'$ and every vertex of $G'$ is a point of $S$.  For a point $u = (a, b) \in \R^2$, we
will write $x(u) = a$ and $y(u) = b$ for the standard projections.

For a plane graph $G$ with $n$ vertices and an ordering of the vertices $\ba = (a_1, \ldots, a_n)$, we define a sequence of plane graphs $G_0(\ba), \ldots,
G_n(\ba)$ recursively by $G_n(\ba) = G$ and $G_{i - 1}(\ba) = G_i(\ba)  - \{a_i\}$.  We will write $G_i$ for $G_i(\ba)$ when $\ba$ is understood.  If $v$ is a
vertex of $G_i$ then we let $d_i(v)$ denote the degree of $v$ \emph{in the graph $G_i$}.

A \emph{plane triangulation} is a $2$-connected plane graph $G$ such that each bounded face of $G$ has exactly $3$ vertices.  Note in particular that if $G$ is a plane triangulation then $\F(G)$ is homeomorphic to a $2$-ball.  A boundary vertex $v$ of a plane triangulation $G$ is a \textit{shedding vertex} of $G$ if $G -
\{v\}$ is a plane triangulation.  Let $G$ be a plane triangulation with $n$ vertices. A vertex sequence $\ba = (a_1, \ldots, a_n)$ is called a \emph{shedding
sequence} for $G$ if $a_i$ is a shedding vertex of $G_i(\ba)$ for all $i = 4, \ldots, n$.  We have the following technical lemma given in \cite[$\S$2]{FPP},
where it was used for an effective embedding of graphs.

\begin{lemma}[\cite{FPP}]
Let $G$ be a plane triangulation.  Then, for every boundary edge $uv$ of $G$,
there is a shedding sequence $\ba = (a_1,\ldots, a_n)$ for~$G$, such that
$u = a_1$ and $v = a_2$.
\label{shedding} \end{lemma}

We say that a \emph{strictly} convex polygon $P \subset \R^2$ with edge $e$ is \emph{projectively convex} with respect to $e$ if $P$ is contained in a triangle
having $e$ as an edge.  A shedding sequence $\ba = (a_1, \ldots, a_n)$ for a plane triangulation $G$ is a \emph{convex shedding sequence} if the region
$\F(G_i(\ba))$ is a projectively convex polygon with respect to the edge $a_1a_2$ for all $i = 3, \ldots, n$.  A geometric embedding $G'$ of $G$ is
\emph{sequentially convex} if $G'$ has a convex shedding sequence.

\bigskip

\section{Drawing the triangulation on a grid}\label{s:tri}

\subsection{A Rational Embedding}
First we address a much easier question: How does one obtain a sequentially convex embedding of $G$ in $\qqq^2$ (that is, with vertex coordinates
\emph{rational})?  We describe a simple construction that produces such an embedding.  The method used to accomplish this easier task will provide part of the
motivation and intuition behind the more involved method we will use to obtain a polynomially sized embedding in $\Z^2$.

\begin{thm}
Let $G$ be a plane triangulation with $n$ vertices and boundary edge~$uv$,
and let $\ba = (a_1, \ldots, a_n)$ be a shedding sequence for $G$ with
$u= a_1$, $v = a_2$.  Then $G$ has a geometric embedding $G'$ in $\qqq^2$,
such that the corresponding sequence $\ba' = (a_1', \ldots, a_n')$ is a convex shedding
sequence for $G'$.
\label{t:rational} \end{thm}

\begin{proof}
To simplify the proof, we make the stronger claim that the edge $a_1'a_2'$ lies along the $x$-axis,
with $a_1'$ to the left of $a_2'$, and $G'$ lies in the upper half-plane.
We proceed by induction on $n$.  If $n = 3$ then we may take the triangle with coordinates
$a_1' = (0, 0)$, $a_2' = (2, 0)$, $a_3' = (1, 1)$ as a
sequentially convex embedding of $G$ in $\qqq^2$.

If $n > 3$, then by the inductive hypothesis there is an embedding $G_{n - 1}'$ of $G_{n - 1}$ in $\qqq^2$ such that $(a_1', \ldots, a_{n - 1}')$ is a convex
shedding sequence for $G_{n - 1}'$, and furthermore $y(a_1') = y(a_2') = 0$, $x(a_1') < x(a_2')$, and $G_{n - 1}'$ lies in the upper half-plane.  Let $w_1,
\ldots, w_k$ denote the neighbors of $a_n$ in $G$, and let $w_1', \ldots, w_k'$ denote the corresponding vertices of $G_{n - 1}'$, ordered from left to right.
If $w_1' \neq a_1'$, then let $z_1'$ denote the left boundary neighbor of $w_1'$.  Similarly, if $w_k' \neq a_2'$, let $z_2'$ denote the right boundary neighbor
of~$w_k'$.

For adjacent vertices $u$ and $v$, we will denote the slope of the edge $uv$ by $s(uv)$.  Similarly, we will denote the slope of a line $\ell$ by $s(\ell)$.
Consider the lines $\ell_1, \ell_2, \ell_3, \ell_4$ spanned by the edges $z_1'w_1'$, $w_1'w_2'$, $w_{k - 1}'w_k'$, and $w_k'z_2'$, respectively.  If $w_1' =
a_1'$, we may take $\ell_1$ to be any non-vertical line passing through $a_1'$, with slope satisfying $s(\ell_1) > s(\ell_2)$.  Similarly, if $w_k' = a_2'$, we
may take $\ell_4$ to be any non-vertical line passing through $v'$, with slope satisfying $s(\ell_3) > s(\ell_4)$.

Let $A_1, A_4 \subset \R^2$ denote the open half-planes below the lines $\ell_1$ and $\ell_4$, respectively, and let $A_2$ and $A_3$ denote the open half-planes
above the lines $\ell_2$ and $\ell_3$, respectively.  Since $\F(G_{n - 1}')$ is projectively convex with respect to $a_1'a_2'$, the slopes of the lines $\ell_i$
must satisfy $s(\ell_1) > s(\ell_2) > s(\ell_3) > s(\ell_4)$.  Thus the region $S = A_1 \cap A_2 \cap A_3 \cap A_4$ is non-empty (see
Figure~\ref{f:newvertexset}).  Since each set $A_i$ is open, the set $S$ is open, so we may choose a rational point in $S$, call it $a_n'$.  For each $j = 1,
\ldots, k$ add a straight line segment  $e_j$ between $a_n'$ and the vertex $w_j'$.  Since $a_n'$ lies in the region above the lines $\ell_2$ and $\ell_3$, each
line segment $e_j$ will intersect $G_{n - 1}'$ only in the vertex~$w_j'$.

\begin{figure}[t!]
 \begin{center}
  \begin{overpic}[width=11cm]{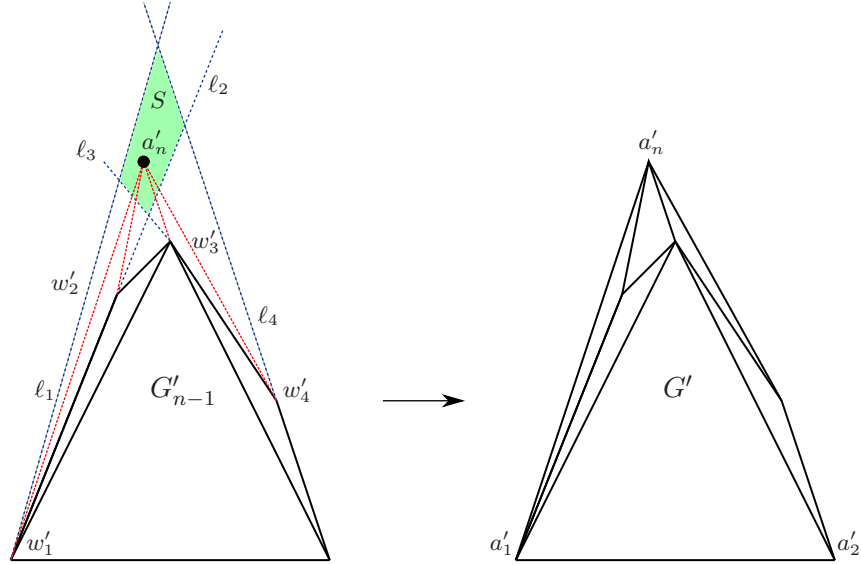}
  	  \put(17, 20){$G_{n - 1}'$}
	  \put(79, 20){$G'$}
	  \footnotesize
  	  \put(16, 50){$a_n'$}
	  \put(3, 20){$\ell_1$}
	  \put(8, 49){$\ell_3$}
	  \put(24, 57){$\ell_2$}
	  \put(30, 29){$\ell_4$}
	  \put(2, 1.5){$w_1'$}
	  \put(5, 33){$w_2'$}
	  \put(22, 38){$w_3'$}
	  \put(33, 20){$w_4'$}
	  \put(58, 1.5){$a_1'$}
	  \put(100, 1.5){$a_2'$}
	  \put(76, 50){$a_n'$}
	  \put(17, 55){$S$}
  \end{overpic}
  \caption{The new vertex $a_n'$, chosen as a rational point of the set $S$.}
  \label{f:newvertexset}
 \end{center}
\end{figure}

Let $G'$ denote the plane graph obtained from $G_{n - 1}'$ by adding the vertex $a_n'$ and the edges~$e_j$.
Then $G'$ is clearly a geometric embedding of $G_n = G$,
such that each vertex $a_i$ corresponds to $a_i'$, for $i = 1, \ldots, n$.
Furthermore, since $a_n'$ lies in the region below the lines $\ell_1$ and
$\ell_4$, the region $\F(G')$ is projectively convex with respect to $a_1'a_2'$.
From this, together with the fact that $(a_1', \ldots, a_{n - 1}')$ is a convex
shedding sequence for $G_{n - 1}'$, we have that $\ba' = (a_1', \ldots, a_n')$
is a convex shedding sequence for $G'$.
\end{proof}

\medskip

\subsection{The Shedding Tree of a Plane Triangulation}
Now we address the problem of embedding the triangulation $G$ on an integer grid.  The idea behind our construction is roughly as follows.  We start with a
triangular base whose the horizontal width is very large.  We then show that, because this horizontal width is large enough, we may add each vertex in a manner
similar to that used in the proof of Theorem~\ref{t:rational}, and we will always have enough room to find an acceptable integer coordinate.  The crucial part of the construction is the careful method in which we add each new vertex.  In particular, there are two distinct methods for adding the new vertex $a_i$, depending on whether $d_i(a_i) = 2$ or $d_i(a_i) > 2$.  To facilitate the proper placement of the vertices $a_i$ with $d_i(a_i) = 2$, we will appeal to a certain tree
structure determined by the shedding sequence $\ba$.  We introduce the following definitions.

Let $G$ be a plane triangulation with shedding sequence $\ba$.  We may assume that $G$ is embedded geometrically as in Theorem~\ref{t:rational}.  Proceeding
recursively, we define a binary tree $T = T(G, \ba)$, such that the nodes of $T$ are edges of $G$, and the edges of $T$ correspond to faces of $G$.

Let $\nu_2$ denote the edge of $G$ containing vertices $a_1, a_2$, and let $T_2$ be the tree consisting of the single node $\nu_2$.  Now let $3 \leq i \leq n$,
and let $\nu_i, \nu_i'$ denote the boundary edges of $G_i(\ba)$ immediately to the left and right of $a_i$, respectively (this is well defined because $G$ is
embedded as in Theorem~\ref{t:rational}).  Assume that we have already constructed $T_{i - 1}$, and that all boundary edges of $G_{i - 1}(\ba)$ are nodes of
$T_{i - 1}$.  Let $\xi, \xi'$ be the boundary edges of $G_{i - 1}(\ba)$ such that $\xi$ shares a face with $\nu_i$, and $\xi'$ shares a face with $\nu_i'$.

Define $T_i = T_i(G, \ba)$ to be the tree obtained from $T_{i - 1}$ by adding $\nu_i$ and $\nu_i'$ as nodes, and adding the edges $(\xi, \nu_i)$ and $(\xi',
\nu_i')$, designated \emph{left} and \emph{right}, respectively.  Then clearly all boundary edges of $G_i(\ba)$ are vertices of $T_i$.  Thus we have a
recursively defined sequence of trees $(T_2, T_3, \ldots, T_n)$, and nodes $(\nu_2, \nu_3, \nu_3', \ldots, \nu_n, \nu_n')$, such that $T_i$ has nodes $\nu_2,
\nu_3, \nu_3', \ldots, \nu_i, \nu_i'$.  We call the trees $T_i(G, \ba)$ the \emph{shedding trees} of $G$, and we write $T = T_n$ (see Figure~\ref{f:tree}).  Note that for all $i = 2, \ldots, n$, we have $T_i(G, \ba) = T(G_i(\ba),
(a_1, \ldots, a_i))$.

Let $\Rc = \{i \in \{1, \ldots, n\} \ | \ d_i(a_i) \leq 2\}$, and define a set of edges \[E_\Rc = \{(\xi, \sigma) \in E(T) \ | \ \text{$\sigma = \nu_j$ or
$\sigma = \nu_j'$ for some $j \notin \Rc$}\}.\]  Let $T_i^\ast = T_i^\ast(G, \ba)$ be the tree obtained from $T_i$ by contracting all edges in $E(T_i) \cap E_\Rc$ (shown in
blue in Figure~\ref{f:tree}).  Note that each $T_i^\ast$ is a \emph{full} binary tree.  We call the trees $T_i^*$ the \emph{reduced
trees} of $G$, and we write $T^\ast = T_n^\ast$.  If $d_i(a_i) = 2$ for all $i \geq 3$, then $T^\ast = T$.

\begin{figure}[t!]
 \begin{center}
  \begin{overpic}[width=11.5cm]{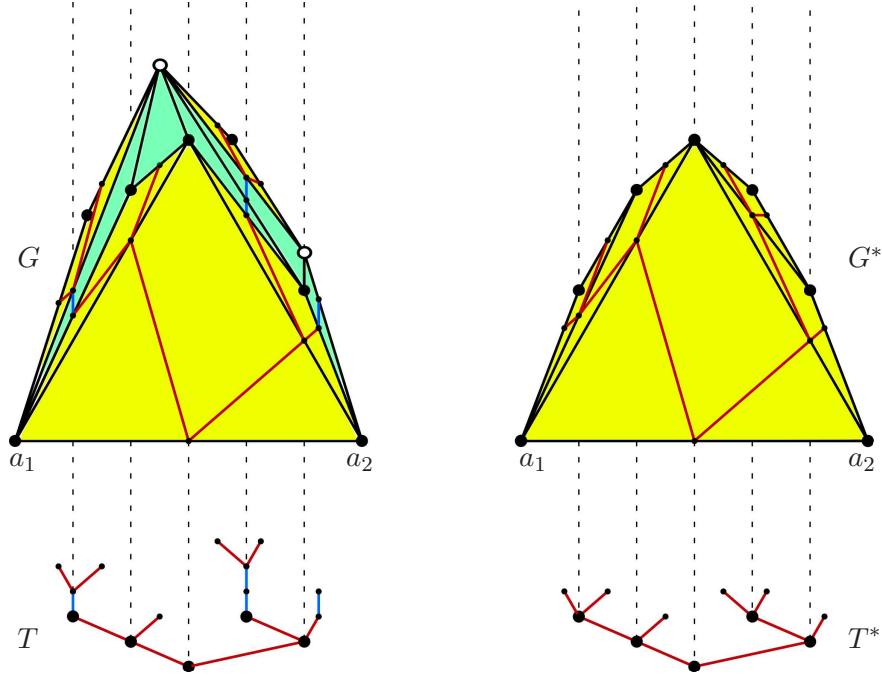}
 	  \put(1, 47){$G$}
	  \put(97, 47){$G^\ast$}
	  \put(1, 3){$T$}
	  \put(97, 3){$T^\ast$}
	  \put(0, 24){$a_1$}
	   \put(39, 24){$a_2$}
	   \put(59, 24){$a_1$}
	   \put(97, 24){$a_2$}
  \end{overpic}
  \caption{The triangulations $G$ and $G^\ast$, together with corresponding trees $T$ and $T^\ast$.  Each node of $T$ corresponds to an edge of $G$, and
  similarly for $T^\ast$ and $G^\ast$.  The tree $T^\ast$ is obtained from $T$ by contracting the blue edges.  The large nodes of $T$ are the internal nodes of
  $T^*$, and correspond to the vertices of $G^*$ other than $a_1$ and $a_2$.}
  \label{f:tree}
 \end{center}
\end{figure}

A fundamental idea behind our integer grid embedding is that the reduced tree $T^\ast$ contains all of the critical information needed for carrying out the
embedding of $G$.  For example, for each $i \in \Rc - \{1, 2\}$, the vertex $a_i$ corresponds to an internal node of $T^\ast$ (shown as large dots in
Figure~\ref{f:tree}).  Thus the structure of $T^\ast$ tells us how to horizontally space the vertices $a_i$ with $d_i(a_i) = 2$, and how to choose the slopes of
the boundary edges adjacent to them.  On the other hand, when adding vertices $a_i$ with $d_i(a_i) > 2$, our construction will have the property that the
boundary slopes will be perturbed only slightly, and the horizontal distances between vertices will only increase.  Furthermore, throughout our entire
construction the total horizontal width of the embedding will remain fixed.

Given a subsequence of $T_2^*, T_3^*, \ldots, T_n^*$ consisting of the \emph{distinct} reduced trees of $G$, there is a natural construction which produces a triangulation $G^*$ and shedding sequence $\ba^\ast$, such that the distinct reduced trees of $G$ are exactly the shedding trees of $G^\ast$.  This construction is the content of the next lemma.  That is, we will use Lemma~\ref{t:degree2} to define~$G^\ast$ precisely in Section~\ref{s:grid}.  This lemma can also be thought of as a special instance of Theorem~\ref{plane}, in the case that $d_i(a_i) = 2$ for all $i \geq 3$.

\begin{lemma}  Let $n \geq 3$, and let $(t_2, \ldots, t_n)$ be a sequence of full binary trees, such that $t_{i - 1}$ is a subtree of $t_i$, and $t_i$ has $1 +
2(i - 2)$ nodes, for all $i = 2, \ldots, n$.  Then there is a sequentially convex plane triangulation $H$ with $n$ vertices, embedded in a $2(n - 2) \times
\binom{n - 1}{2}$ integer grid, and a convex shedding sequence $\ba$ for $H$, such that $t_i$ is isomorphic to $T_i(H, \ba)$ for all $\nobreak{i = 2, \ldots, n}$.
\label{t:degree2}
\end{lemma}

\begin{proof}
	Let $m$ and $m'$ denote the number of internal nodes of $t_n$ to the left and right of the root node, respectively.  Note that $m + m' + 3 = n$.  Without
loss of generality we may assume $m \leq m'$.  For $-m \leq k \leq m' + 1$, we define \[x_k = k, \quad \quad \quad y_k = \binom{m' + 2}{2} - \binom{|k| +
1}{2}.\]  Additionally, we define \[x_{-m - 1} = -x_{m' + 1}, \quad \quad \quad y_{-m - 1} = 0.\]  Note that the $n$ points $(x_k, y_k)$ all lie on the (piecewise)
parabola defined by
\[y \. = \. -\ts \frac{x^2+|x|}{2} \. + \. \frac{(m' + 2)(m' + 1)}{2}.\]
These points will serve as the vertices of the triangulation~$H$
(compare with the vertices of~$G^*$ in Figure~\ref{f:tree}).
		
	For all $i = 3, \ldots, n$, since $t_i$ is full and contains two more nodes than $t_{i - 1}$, it follows that $t_i$ contains exactly one more
internal node than $t_{i - 1}$.  Let $(\xi_3, \ldots \xi_n)$ denote the sequence of internal nodes so obtained.  Note that $\xi_3$ is the root node of all the
trees $t_i$, and $t_2$ consists of the single node $\xi_3$.  The nodes of $t_n$ may be linearly ordered by a depth-first search on $t_n$, such that left nodes
are visited before right nodes.  Call this order~$\Omega$.  This restricts to a linear order on the internal nodes of $t_n$, which induces a permutation $\omega
: \{3, \ldots, n\} \rightarrow \{3, \ldots, n\}$.  That is, node $\xi_i$ has position $\omega(i)$ in the order $\Omega$.  Then we define a sequence of points
$\ba = (a_1, a_2, \ldots, a_n)$ by
\[ \begin{array}{cclc}
a_1  & = & (x_{-m - 1}, y_{-m - 1}),& \\ a_2 & =  & (x_{m' + 1}, y_{m' + 1}), &\\ a_i  & =  & (x_{\omega(i) - m - 3}, y_{\omega(i) - m - 3}) & \text{for $3 \leq
i \leq n$.}
\end{array} \]

For $\nobreak{i = 3, \ldots, n}$, let $\nobreak{\omega_i : \{3, \ldots, i\} \rightarrow \{3, \ldots, i\}}$ denote the permutation induced by restricting the order $\Omega$
to the internal nodes $(\xi_3, \ldots, \xi_i)$ of $t_i$.  Note that $\omega = \omega_n$.  For each $\nobreak{i = 3, \ldots, n}$, we would like to determine the internal
nodes of $t_i$ that immediately precede and succeed $\xi_i$ in the order $\omega_i$.  For this purpose, we define functions \[f, g : \{3, \ldots, n\}
\rightarrow \{1, \ldots, n\}\] as follows.
\begin{align*}
f(i) & = \begin{cases}
	\omega_i^{-1}(\omega_i(i) - 1) & \text{if $\omega_i(i) > 3$}, \\
	1 & \text{otherwise}
	\end{cases} \\
g(i) & = \begin{cases}
	\omega_i^{-1}(\omega_i(i) + 1) & \text{if $\omega_i(i) < i$}, \\
	2 & \text{otherwise.}
	\end{cases}
\end{align*}

We may now define a sequence of plane triangulations $H_1, \ldots, H_n$ recursively.  Let $H_1$ consist of the single vertex $a_1$, and let $H_2$ consist of the
vertices $a_1$, $a_2$ and the line segment $a_1a_2$.  Now let $3 \leq i \leq n$, and suppose we have constructed $H_{i - 1}$.  We obtain $H_i$ by adding the
vertex $a_i$ and the line segments $a_ia_{f(i)}$ and $a_ia_{g(i)}$ to $H_{i - 1}$.  This completes the construction of the graphs $H_2, \ldots, H_n$.  We write
$H = H_n$.

We now check that $\ba = (a_1, \ldots, a_n)$ is a convex shedding sequence for $H$.  By construction, we have immediately that $H_i$ is a plane triangulation
with $H_{i - 1} = H_i - \{a_i\}$ for all $\nobreak{i = 2, \ldots, n}$, and furthermore $d_i(a_i) = 2$ for all $i \geq 3$.  For $k \geq -m$, the slope of the edge between
adjacent vertices $(x_k, y_k)$ and $(x_{k + 1}, y_{k + 1})$ of $H$ is
\[\frac{y_{k + 1} - y_k}{x_{k + 1} - x_k} = \binom{|k| + 1}{2} - \binom{|k + 1| + 1}{2} =
\begin{cases}
	-(k + 1) & \text{if $k \geq 0$}, \\
	-k & \text{if $k < 0$.}
\end{cases}
\]
Additionally, the slope of the edge between $(x_{-m - 1}, y_{-m - 1})$ and $(x_{-m}, y_{-m})$ is
\[
\begin{aligned}
& \frac{y_{-m} - y_{-m-1}}{x_{-m} - x_{-m-1}} = \frac{y_{-m}}{-m + (m' + 1)} = \frac{1}{m'  - m + 1}\left[\binom{m' + 2}{2} - \binom{m + 1}{2}\right] \\
&= \frac{(m' + 2)(m' + 1) - (m + 1)m}{2(m' - m + 1)} = \frac{(m' - m + 1)(m' + m + 2)}{2(m' - m + 1)} \\
&= \frac{m' + m + 2}{2} \geq \frac{m + m + 2}{2} = m + 1.
\end{aligned}
 \]
Thus the boundary edge slopes of $H$ are strictly decreasing from left to right.  Since $\nobreak{d_i(a_i) = 2}$ for all $i \geq 3$, this implies that the
boundary edge slopes of each $H_i$ are also strictly decreasing from left to right.  It follows that $\F(H_i)$ is projectively convex, for all $i \geq 3$.  Hence
$\ba$ is a convex shedding sequence for $H$.

To see that $t_i$ is isomorphic to $T_i(H, \ba)$ for all $i = 2, \ldots, n$, we construct an explicit isomorphism.  We define a map $\psi_2 : t_2 \rightarrow
T_2(H, \ba)$ by $\psi_2(\xi_3) = a_{f(3)}a_{g(3)} = a_1a_2$, which is trivially an isomorphism.  For $i \geq 3$, and $j = 3, \ldots, i$, let $\xi_j^-$ and
$\xi_j^+$ denote the left and right child, respectively, of the internal node $\xi_j$ of $t_i$.  We define a map $\psi_i : t_i \rightarrow T_i(H, \ba)$ by
\[ \begin{array}{lclr}
\psi_i(\xi_j)  & =  & a_{f(j)}a_{g(j)}, & \\ \psi_i(\xi_j^-) & = & a_ja_{f(j)}, & \\ \psi_i(\xi_j^+) & =  & a_ja_{g(j)} & \quad \text{for $j = 3, \ldots, i$}.
\end{array} \]

From the definition of the order $\Omega$ and the resulting functions $f$ and $g$, it is straightforward to check that $\psi_i$ is well-defined and bijective.
Since the triangle $(a_ja_{f(j)}a_{g(j)})$ is a face of $H_j$ for all $j = 3, \ldots, i$, the pairs $(a_{f(j)}a_{g(j)}, a_ja_{f(j)})$ and $(a_{f(j)}a_{g(j)},
a_ja_{g(j)})$ are edges of $T_i(H, \ba)$.  Thus $\psi_i$ is clearly a tree isomorphism.  We may think of $\psi_n$ as providing a correspondence between the
internal node $\xi_i$ and the vertex~$a_i$ (whose neighbors in $H_i$ are $a_{f(i)}$ and $a_{g(i)}$), for all $i = 3, \ldots, n$ (see Figure~\ref{f:tree}).

Finally, the width of the grid is \[x_{m' + 1} - x_{-m - 1} = 2x_{m' + 1} = 2(m' + 1) \leq 2(n - 2),\] and the height of the grid is \[y_0 = \binom{m' + 2}{2}
\leq \binom{n - 1}{2}.\]  Therefore $H$ is embedded in an integer grid of size $2(n - 2) \times \binom{n - 1}{2}$.
\end{proof}

\medskip

\subsection{The Integer Grid Embedding}\label{s:grid}
Given a plane triangulation $G$ with $n$ vertices and shedding sequence $\ba$, let $t_i = T_i^\ast(G, \ba)$ denote the reduced trees of $G$.  Let $\rho$ denote
the unique increasing bijection from $\Rc$ to $\{1, \ldots, R\}$, where $\Rc$ is the subset of $\{1, \ldots, n\}$ defined above, and $R = |\Rc|$.  Note that
$\nobreak{1, 2, 3 \in \Rc}$, so $\rho(1) = 1, \rho(2) = 2, \rho(3) = 3$.  Define a map $\nobreak{h : \{1, \ldots, n\} \rightarrow \{1, \ldots, R\}}$ by taking
$h(i)$ to be the unique index for which \[\rho^{-1}(h(i)) \leq i < \rho^{-1}(h(i) + 1).\] (We require the inequality on the right to hold only when $h(i) + 1 \leq R$.)  In particular, if $i \in \Rc$ then $h(i) = \rho(i)$.

The sequence of distinct trees $t_{\rho^{-1}(1)}, \ldots, t_{\rho^{-1}(R)}$ satisfies the hypotheses of Lemma~\ref{t:degree2}.  Therefore we let $G^\ast$
denote the sequentially convex triangulation constructed from this sequence of trees as in Lemma~\ref{t:degree2}.  That is, $G^\ast$ is the triangulation $H$ in the notation of the lemma, and $G^\ast$ has the exact vertex coordinates given in the lemma.  We let $\ba^\ast = (a_1^\ast, \ldots, a_R^\ast)$ denote the corresponding convex shedding sequence of $G^\ast$ produced by Lemma~\ref{t:degree2}.  We also define $G_i^\ast = G^\ast_i(\ba^\ast)$ for all $i = 1, \ldots, R$, so in particular $G^\ast_R = G^\ast$. We call the $G_i^\ast$ the \emph{reduced
triangulations} of $G$ (see Figure~\ref{f:tree}).  Note that each vertex $a_i^\ast$ has degree $2$ in $G_i^\ast$.  So we may think of $G^\ast$
as being obtained from $G$ by ``throwing away" all vertices $a_i$ for which $d_i(a_i) > 2$.  It was this property that originally motivated our definition of
$G^\ast$.

As we will see, the triangulation $G^\ast$ will tell us exactly how to add vertices of degree $2$, in our construction of a sequentially convex embedding of $G$.
A particular property of the reduced triangulations makes this possible.  Namely, for any boundary edge $e$ of $G_i$, there is a \emph{corresponding boundary
edge} $e^\ast$ of $G_{h(i)}^\ast$, which we define as follows.  First note that from the definitions, the trees
$T_{h(i)}(G^\ast, \ba^\ast)$ and $T_i^\ast(G, \ba)$ are isomorphic.  Thus we may think of an edge of $G_{h(i)}^\ast$ (which is a node of $T_{h(i)}(G^\ast, \ba^\ast)$) as a node
of $T_i^\ast(G, \ba)$.  So for a boundary edge $e$ of $G_i$, we denote by $e^\ast$ the unique boundary edge of $G_{h(i)}^\ast$ (thought of as a node of $T_i^\ast(G,
\ba)$), that is identified with the node $e$ of $T_i(G, \ba)$ upon contracting the edges in the set~$E_\Rc$.

\begin{thm}
Let $G$ be a plane triangulation with $n$ vertices and boundary edge $uv$,
and let $\ba = (a_1, \ldots, a_n)$ be a shedding sequence for~$G$ with
$u = a_1$, $v = a_2$.  Then $G$ has a geometric embedding~$G'$ in a
$4n^3 \times 8n^5$ integer grid, such that the corresponding sequence
$\ba' = (a_1', \ldots,a_n')$ is a convex shedding sequence for~$G'$.
\label{plane} \end{thm}

\begin{proof}
We recursively construct a sequence of graphs $G_1', \ldots, G_n'$, and a sequence of vertices $a_1', \ldots, a_n'$, such that each $G'_i$ is a geometric
embedding of $G_i$ with convex shedding sequence $\ba' = (a_1', \ldots, a_i')$, where $a_i'$ is the vertex of $G'$ corresponding to $a_i$.  Let $G_i^*$ denote
the reduced triangulations of $G$, and let $\ba^\ast = (a_1^\ast, \ldots, a_{R}^\ast)$ denote the corresponding shedding sequence for $G^\ast$.  Let $m'$
denote the number of vertices of $G^\ast$ lying between $a_3^\ast$ and $a_2^\ast$, and $m$ the number of vertices lying between $a_1^\ast$ and $a_3^\ast$.  Then
$m' + m + 3 = n$.  Without loss of generality we may assume that $m \leq m'$.

For points $v_1, v_2 \in \R^2$ and $e = v_1 v_2$ the line segment between them, we write \[x(e) = |x(v_1) - x(v_2)| \quad \quad \text{and} \quad \quad y(e) =
|y(v_1) - y(v_2)|.\]  We first scale $G^\ast$ to obtain a much larger triangulation, which we will use to construct the triangulations $G_i'$.  We choose the
scaling factors large enough so that we will have ``enough room" to carry out our constructions.  Specifically, let $\alpha = 2n^2 + n + 1$ and $\beta =
2n\alpha$.  We define $Z_i$ to be the result of scaling $G_i^*$ by a factor of $\alpha$ in the $x$ dimension and $\beta$ in the $y$ dimension.  That is, for each
$i = 1, \ldots, R$, we define $z_i = (\alpha x(a_i^\ast), \beta y(a_i^\ast))$.  Then $\bz = (z_1, \ldots, z_R)$ is the shedding sequence for $Z_R$ corresponding
to $\ba^\ast$.  We write $Z = Z_R$.

Since $G_i' \sim G_i$ (as we verify below) the edges of $G_i'$ and $G_i$ are in correspondence.  Thus every boundary edge $e$ of $G_i'$ has a \emph{corresponding boundary edge} $e^\ast$ of $G_{h(i)}^\ast$, as defined above.  We then write $Z(e)$ for the edge of $Z_{h(i)}$ corresponding to $e^\ast$.  Note that if $e^\ast$ has slope $s$,
then $Z(e)$ has slope $\frac{\beta}{\alpha}s$.  In particular, since $m' + 1$ is the largest magnitude of the slope of any edge of $G^\ast$, we see that
$\frac{\beta}{\alpha}(m' + 1) = 2n(m' + 1)$ is the largest magnitude of the slope of any edge of $Z$.  Let $M$ denote this slope, and note that $M \leq 2n^2$.
Note also that the absolute difference of two boundary edge slopes of $Z$ is at least $\frac{\beta}{\alpha} = 2n$.  It follows immediately that for each $i = 1,
\ldots, R$, the absolute difference of two boundary edge slopes of $Z_i$ is at least~$2n$.

Define $a_1' = z_1$ and $a_2' = z_2$.  Take $G_1'$ to consist of the single vertex $a_1'$, and take $G_2'$ to consist of the vertices $a_1', a_2'$, together with the line segment $a_1'a_2'$.  Now let $3 \leq i \leq n$, and suppose we have constructed $G_{i - 1}'$.  To define $a_i'$, we consider two cases, namely whether
$d_i(a_i) = 2$ or $d_i(a_i) > 2$.

\begin{figure}[t!]
 \begin{center}
  \begin{overpic}[height=4cm]{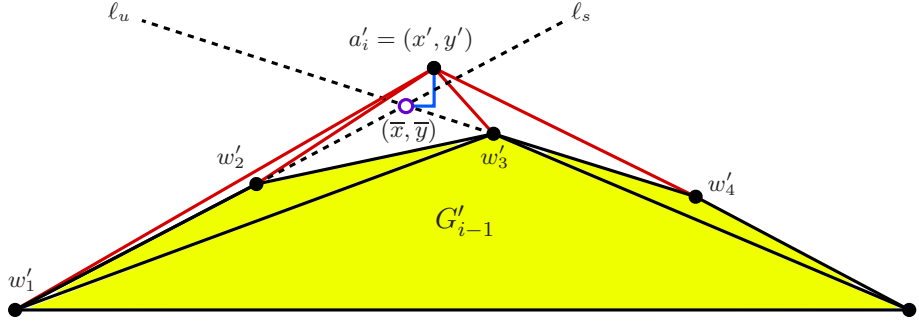}
  	 \put(47, 10){$G_{i - 1}'$}
	 \footnotesize
	  \put(41, 20){$(\overline{x}, \overline{y})$}
	  \put(37.5, 30){$a_i' = (x', y')$}
	  \put(0, 3.5){$w_1'$}
	  \put(23, 17){$w_2'$}
	  \put(52, 17){$w_3'$}
	  \put(77, 14){$w_4'$}
	  \put(11, 33){$\ell_u$}
	  \put(62, 33){$\ell_s$}
  \end{overpic}
  \caption{The construction of vertex $a_i'$ when $d_i(a_i) > 2$.  In this example, $d_i(a_i) = 4$.}
  \label{f:highdegree}
 \end{center}
\end{figure}

\medskip

\noindent
\textbf{Construction of $a_i'$, in the case $d_i(a_i) > 2$}. \ts If $d_i(a_i) > 2$, then let $w_1, \ldots, w_k$ denote the neighbors of $a_i$ in $G_i$,
and let $w_1', \ldots, w_k'$ denote the corresponding vertices of $G_{i - 1}'$, ordered from left to right.  Let $s$ denote the slope of the edge $w_1'w_2'$, and
$u$ the slope of the edge $w_{k - 1}'w_k'$.  Let $\ell_s$ denote the line of slope $s$ containing the point $w_1'$, and $\ell_u$ the line of slope $u$ containing
the point $w_k'$.  We denote by $(\overline{x}, \overline{y})$ the point of intersection of the lines~$\ell_s$ and~$\ell_u$.  Let $x' = \lceil \overline{x}
\rceil$ and $\gamma = x' - \overline{x}$, and let $y' = \lceil \overline{y} \rceil + \lfloor \gamma s \rfloor + 1$.  We now define $a_i' = (x', y')$ (see
Figure~\ref{f:highdegree}).  We obtain $G_i'$ from $G_{i - 1}'$ by adding the vertex $a_i'$, together with all line segments between $a_i'$ and the vertices
$w_1', \ldots, w_k'$.

\medskip

\noindent
\textbf{Construction of $a_i'$, in the case $d_i(a_i) = 2$}. \ts If $d_i(a_i) = 2$, then let $\Delta$ be the triangle of $Z_i$ containing $z_{\rho(i)}$.
Let $w_1, w_2$ denote the boundary neighbors of $a_i$ in $G_i$, and let $w_1', w_2'$ denote the corresponding vertices of $G_{i - 1}'$, so that $w_1'$ lies to
the left of $w_2'$.  We are going to construct a triangle $\Delta'$, such that $\Delta'$ is the image of $\Delta$ under an affine map which is the composition of
a uniform scaling and a translation.  Furthermore, we will place $\Delta'$ in a specific position with respect to the triangulation $G_{i - 1}'$.  In particular,
if $v_1, v_2, v_3$ denote the vertices of $\Delta'$, we require that $x(v_1) = x(w_1')$, $v_2 = w_2'$, and $x(v_1) < x(v_3) < x(v_2)$ (see
Figure~\ref{f:degree2}).  It is easily verified that these conditions, together with the requirement that $\Delta'$ is a scaled, translated copy of $\Delta$,
determine the vertices $v_1, v_2, v_3$ of $\Delta'$ uniquely.

To define the new vertex $a_i'$, we start by applying a vertical shearing to the triangle $\Delta'$, namely the unique shearing that fixes $v_2 = w_2'$ and maps~$v_1$ to~$w_1'$.  We will denote the image of $v_3$ under this shearing by $\overline{v_3}$.  So in terms of the vertices of $Z_{\rho(i)}$ and $G_{i - 1}'$, the point $\overline{v_3}$ is defined as follows.

Let $\eta = y(v_1) - y(w_1')$.  Let $b_1$ and $b_2$ denote the left and right boundary neighbors, respectively, of $z_{\rho(i)}$ in $Z_{\rho(i)}$.
Notice that the edge $b_1b_2 = Z(w_1'w_2')$ is a boundary edge of $Z_{h(i - 1)}$.  We will define a ratio $\kappa$, which describes how far away
$z_{\rho(i)}$ is from $b_2$, in the $x$ direction.  That is, we define \[\kappa = \frac{x(b_2) - x(z_{\rho(i)})}{x(b_2) - x(b_1)}.\]  We may now define
$\overline{v_3} = (x(v_3), y(v_3) - \kappa \eta)$.

Note that $z_3$ is the apex of $Z_i$ for all $i = 1, \ldots, R$, and $x(z_3) = \alpha x(a_3^\ast) = \alpha \cdot 0 = 0$.  So if $x(z_{\rho(i)}) < 0$ then
$x(z_{\rho(i)})$ lies to the left of the apex of $Z_i$, and if $x(z_{\rho(i)}) > 0$ then $z_{\rho(i)}$ lies to the right of the apex.  With this understanding,
we let
\[v_3' = \begin{cases}
(\lfloor x(\overline{v_3}) \rfloor, \lceil y(\overline{v_3}) \rceil) & \text{if $x(z_{\rho(i)}) \leq 0$,} \\
(\lceil x(\overline{v_3}) \rceil, \lceil y(\overline{v_3}) \rceil) & \text{if $x(z_{\rho(i)}) > 0$.}
\end{cases}\]
We now define $a_i' = v_3'$.  We obtain $G_i'$ from $G_{i - 1}'$ by adding the vertex $a_i'$ and the two line segments between $a_i'$ and the vertices $w_1',
w_2'$.

\begin{figure}[t!]
 \begin{center}
  \begin{overpic}[height=6cm]{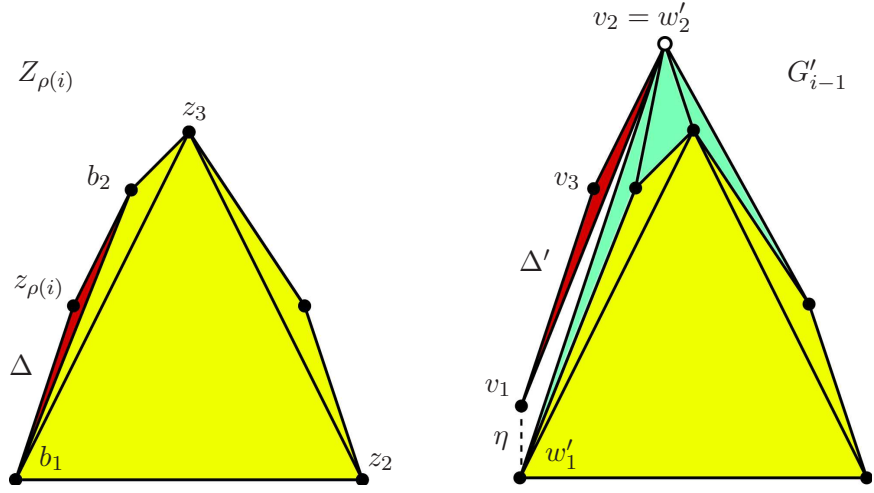}
  	\put(1, 47){$Z_{\rho(i)}$}
	 \put(90, 47){$G_{i - 1}'$}
	  \put(.5, 23){$z_{\rho(i)}$}
	  \put(3.5, 2.5){$b_1$}
	  \put(41.5, 2.5){$z_2$}
	  \put(20, 43){$z_3$}
	  \put(9, 35){$b_2$}
	  \put(55, 11){$v_1$}
	  \put(67.5, 53.5){$v_2 = w_2'$}
	  \put(63, 35){$v_3$}
	  \put(56, 5){$\eta$}
	  \put(62, 3){$w_1'$}
	  \put(0, 13){$\Delta$}
	  \put(59, 25){$\Delta'$}
  \end{overpic}
  \caption{The first stage of the construction of vertex $a_i'$ when $d_i(a_i) = 2$.  The red triangles $\Delta$ and $\Delta'$ differ by a uniform scaling and a
  translation.}
  \label{f:degree2}
 \end{center}
\end{figure}

\medskip

\noindent
\textbf{Verification of the construction.} \ts
We have now explicitly described the construction.  It remains to show that the above constructions
actually produce a \emph{convex} shedding sequence $\ba' = (a_1', \ldots, a_n')$ for $G'$, and that $G'$ lies in the grid size indicated.  Clearly, the horizontal
width of the grid remains constant throughout the construction.  Specifically, the triangulations $G_1', G_2', \ldots, G_n'$ all have the same width $\alpha 2(n
- 2)$, which is the width of the $Z_i$.  So to show that the construction is sequentially convex, and that the bound on the height of $G'$ is correct, we will
calculate how the boundary slopes are modified when we add the new vertex $a_i'$ in the two cases $d_i(a_i) = 2$ and $d_i(a_i) > 2$.

For each $3 \leq i \leq n$, let \textbf{$\cP(i)$} denote the conjunction of the following three properties, that we wish to show:
\begin{align*}
\P(i, 1). & \quad \text{For every boundary edge $e$ of $G_i'$, we have $x(e) \geq x(Z(e))$.} \\
\cP(i, 2). & \quad \text{For every boundary edge $e$ of $G_i'$, the slopes of $e$ and $Z(e)$ differ by at most $i$.} \\
\cP(i, 3). & \quad \text{$G_i' \sim G_i$ and $G_i'$ has convex shedding sequence $(a_1', \ldots a_i')$.}
\end{align*}
To prove $\cP(i)$ for each $i = 3, \ldots, n$, we proceed by induction on $i$.

For $i = 3$, note that $a_1' = z_1$, $a_2' = z_2$, and the vertex $a_3'$ is chosen so that in particular the triangle $(a_1'a_2'a_3')$ is a scaling of the
triangle $(z_1z_2z_3) = \F(Z_3)$.  This implies that $a_3' = z_3$.  Thus $G_3' = Z_3$, which immediately establishes $\cP(3)$.

Now let $i > 3$, and suppose that $\cP(i - 1)$ holds.  As in the construction, we consider separately the cases $d_i(a_i) > 2$ and $d_i(a_i) = 2$.

\medskip

\noindent
\textbf{Verification of $\cP(i)$ in the case $d_i(a_i) > 2$.} \ts Clearly, we have $Z(w_1'a_i') = Z(w_1'w_2')$ and $Z(a_i'w_k') = Z(w_{k - 1}'w_k')$.  Note that
$x(w_2') \leq \overline{x} \leq x(w_{k - 1}')$, and thus $\nobreak{x(w_2') \leq x(a_i') \leq x(w_{k - 1}')}$.    Therefore
\[x(w_1'a_i') = x(a_i') - x(w_1') \geq x(w_2') - x(w_1') = x(w_2'w_1') \geq x(Z(w_2'w_1')) = x(Z(w_1'a_i')),\] where the last inequality follows from $\cP(i - 1,1)$.  Similarly, we have $x(a_i'w_2') \geq x(Z(a_i'w_2'))$.  Thus $\cP(i, 1)$ holds.

We now show that the slopes of the boundary edges of $G_i'$ containing $a_i'$ differ only slightly from the slopes $s$ and $u$ defined in the above construction. That is, let $s'$ denote the slope of the line passing through $w_1'$ and $a_i' = (x', y')$, and let $u'$ denote the slope of the line passing through $a_i'$ and $w_k'$.  Since \[y'  - \overline{y} = (\lceil \overline{y} \rceil - y) + \lfloor \gamma s \rfloor + 1 \geq \lfloor \gamma s \rfloor + 1 > \gamma s,\] we clearly
have $s' - s > 0$.  On the other hand,
\[
\begin{aligned}
 & s' - s =  \dfrac{y' - y(w_1')}{\overline{x} + \gamma - x(w_1')} - s =  \dfrac{y' - y(w_1') - (\overline{x} - x(w_1'))s - \gamma s}{\overline{x} + \gamma -
 x(w_1')}  = \dfrac{(y' - \overline{y}) - \gamma s}{\overline{x} + \gamma - x(w_1')} \\
 &= \dfrac{(\lceil \overline{y} \rceil - y) + (\lfloor \gamma s \rfloor - \gamma s)  + 1}{\overline{x} + \gamma - x(w_1')} <  \dfrac{2}{\overline{x} + \gamma -
 x(w_1')}  \leq  \dfrac{2}{\overline{x} - x(w_1')}  \leq  \dfrac{2}{x(w_2') - x(w_1')} \\
&\leq  \dfrac{2}{x(Z(w_2'w_1'))} \leq \dfrac{2}{\alpha}  \leq  1.
 \end{aligned}
\]
In the last line we have used $\cP(i - 1, 1)$.

Since $s > u$, we have $y' - \overline{y} > \gamma s > \gamma u$.  From this, together with the fact that $x' \geq \overline{x}$, it follows that $u - u' > 0$.
By $\cP(i - 1, 2)$, we have $|s| \leq M + (i - 1)$.  Thus
\[
\begin{aligned}
& u - u' = \dfrac{y(w_k') - \overline{y}}{x(w_k') - \overline{x}} - \dfrac{y(w_k') - y'}{x(w_k') - x'}  \leq  \dfrac{y(w_k') - \overline{y}}{x(w_k') - x'} -
\dfrac{y(w_k') - y'}{x(w_k') - x'} =  \dfrac{y' - \overline{y}}{x(w_k') - x'} \\
&=  \dfrac{(\lceil \overline{y} \rceil - \overline{y}) + \lfloor \gamma s \rfloor + 1}{x(w_k') - x'} < \dfrac{|s| + 2}{x(w_k') - x'} \leq \dfrac{|s| + 2}{x(w_k')
- x(w_{k - 1}')} \leq \dfrac{|s| + 2}{x(Z(w_{k - 1}'w_k'))} \\
&\leq \dfrac{|s| + 2}{\alpha} \leq \dfrac{M + (i - 1) + 2}{\alpha} \leq \dfrac{2n^2 + i + 1}{\alpha} \leq \dfrac{2n^2 + n + 1}{\alpha} = 1.
\end{aligned}
\]
In the second line we have used $\cP(i - 1, 1)$.

Let $Z(s)$ denote the slope of the edge $Z(w_1'w_2')$, and let $Z(u)$ denote the slope of $Z(w_{k - 1}'w_k')$.  By $\cP(i - 1, 2)$, we have $|s - Z(s)| \leq i -
1$ and $|u - Z(u)| \leq i - 1$.  Thus
\[|s' - Z(s)| \leq |s' - s| + |s - Z(s)| \leq 1 + (i - 1) = i,\] and similarly $|u' - Z(u)| \leq i$, so $\cP(i, 2)$ holds.

Since  $s' - s > 0$ and $u - u' > 0$, each line segment $a_i'w_j'$ intersects $G_{i - 1}'$ only in the vertex $w_j'$, for all $j = 1, \ldots k$.  Thus $G_i'$ is
a plane triangulation, and $G_i' \sim G_i$.  It remains to show that $\F(G_i')$ is projectively convex, in order to establish $\cP(i, 3)$.  To do this, we will
show that when the slope $s$ is changed to $s'$ for example, convexity is preserved at the vertex $w_1'$.  That is, the slope $s'$, while greater than $s$, is
still less than the slope of the boundary edge to the left of $w_1'w_2'$.

Let $\hat{s}$ denote the slope of the boundary edge of $G_i'$ adjacent and to the left of $w_1'$, if such an edge exists, and let $\hat{u}$ denote the slope of
the boundary edge of $G_i'$ adjacent and to the right of $w_k'$, if such an edge exists.  Let $Z(\hat{s})$ and $Z(\hat{u})$ denote the boundary slopes of $Z_{h(i
- 1)}$ corresponding to $\hat{s}$ and $\hat{u}$, respectively.  By $\cP(i - 1, 2)$, we have $s - Z(s) < i - 1$ and $Z(\hat{s}) - \hat{s} < i - 1$.  Thus
\[
\begin{aligned}
&\hat{s} - s = (\hat{s} - Z(s)) - (s - Z(s)) \geq (\hat{s} - Z(s)) - (i - 1) \\
&= (Z(\hat{s}) - Z(s)) - (Z(\hat{s}) - \hat{s}) - (i - 1) \geq (Z(\hat{s}) - Z(s)) - 2(i - 1) \\
& \geq 2n - 2i + 2 \geq 2.
\end{aligned}
\]
We conclude that
\[\hat{s} - s' = (\hat{s} - s) - (s' - s) \geq (\hat{s} - s) - 1 \geq 2 - 1 = 1 > 0.\]

An analogous calculation shows that $u' - \hat{u} > 0$.  Thus $w_1'$ and $w_k'$ are convex vertices of $G_i'$.  Because the region $\F(G_{i - 1}')$ is
projectively convex by $\cP(i - 1, 3)$, we conclude that $\F(G_i')$ is projectively convex.  The sequence $(a_1', \ldots, a_{i - 1}')$ is a convex shedding
sequence for $G_{i - 1}'$ by $\cP(i - 1, 3)$, hence $(a_1', \ldots, a_i')$ is a convex shedding sequence for $G_i'$.  Thus $\cP(i, 3)$ holds.  We have now
established $\cP(i)$ in the case that $d_i(a_i) > 2$.

\medskip

\noindent
\textbf{Verification of $\cP(i)$ in the case $d_i(a_i) = 2$.} \ts
First note that by $\cP(i - 1, 1)$, we have \[x(v_1v_2) = x(w_1'w_2') \geq
x(Z(w_1'w_2')) = x(b_1b_2).\]  This implies that $x(w_1'\overline{v_3}) = x(v_1v_3) \geq x(b_1z_{\rho(i)})$, because $\Delta'$ is a scaled, translated copy of
$\Delta$.  Since $x(a_i') = \lfloor x(\overline{v_3}) \rfloor$, and $x(b_1z_{\rho(i)})$ and $x(w_1')$ are integers, we also have \[x(w_1'a_i') \geq
x(b_1z_{\rho(i)}) = x(Z(w_1'a_i')).\]  Similarly, we obtain $x(a_i'w_2') \geq x(z_{\rho(i)}b_2) = x(Z(a_i'w_2'))$.  Thus $\cP(i, 1)$ holds.

We now consider two relevant slopes in the construction of $G_i'$.  Let $r$ denote the slope of the edge $w_1'w_2'$ of $G_{i - 1}'$, and let $Z(r)$ denote the
slope of the corresponding edge $Z(w_1'w_2') = b_1b_2$ of $Z_{h(i - 1)}$.  Since the triangle $\Delta'$ is a scaled, translated copy of $\Delta$, we see that
$Z(r)$ is also the slope of the edge $v_1v_2$ of~$\Delta'$.  Let $\varepsilon = r - Z(r)$, and note that

\[\varepsilon = \frac{y(w_2') - y(w_1')}{x(w_2') - x(w_1')} - \frac{y(v_2) - y(v_1)}{x(v_2) - x(v_1)} = \frac{y(w_2') - y(w_1')}{x(v_2) - x(v_1)} - \frac{y(w_2')
- y(v_1)}{x(v_2) - x(v_1)} = \frac{\eta}{x(v_2) - x(v_1)}.\]

We consider another important pair of slopes arising in the construction of $G_i'$, together with the corresponding pair of slopes of $Z_{\rho(i)}$.
Specifically, let $q_1$ and $q_2$ denote the slopes of the line segments $v_1v_3$ and $v_3v_2$, respectively.  Since $\Delta'$ is a scaled, translated copy of
$\Delta$, these slopes $q_1$ and $q_2$ are also the slopes of the boundary edges $b_1z_{\rho(i)}$ and $z_{\rho(i)}b_2$ of $Z_{\rho(i)}$, respectively.  Let
$\overline{q_1}$ and $\overline{q_2}$ denote the slopes of the line segments $\overline{v_3}w_1'$ and $\overline{v_3}w_2'$, respectively.  We may think of
$\overline{q_1}$ and $\overline{q_2}$ as \emph{modifications} of the slopes $q_1$ and $q_2$, which arise from replacing the vertices $v_1$ and $v_3$ of $\Delta'$
with the vertices $w_1'$ and $\overline{v_3}$, respectively.  A consequence of our definitions is that $\overline{q_1} - q_1 = \overline{q_2} - q_2 =
\varepsilon$.  Indeed,
\[
\begin{aligned}
& \overline{q_1} - q_1 =  \frac{y(\overline{v_3}) - y(w_1')}{x(\overline{v_3}) - x(w_1')} - \frac{y(v_3) - y(v_1)}{x(v_3) - x(v_1)} = \frac{y(\overline{v_3}) -
y(w_1')}{x(v_3) - x(v_1)} - \frac{y(v_3) - y(v_1)}{x(v_3) - x(v_1)} \\
&= \frac{(y(v_1) - y(w_1')) + (y(\overline{v_3}) - y(v_3))}{x(v_3) - x(v_1)} =  \frac{\eta - \kappa \eta}{x(v_3) - x(v_1)} = (1 - \kappa)\frac{\eta}{x(v_3) -
x(v_1)} \\
&= \frac{x(z_{\rho(i)}) - x(b_1)}{x(b_2) - x(b_1)} \cdot \frac{\eta}{x(v_3) - x(v_1)} = \frac{x(v_3) - x(v_1)}{x(v_2) - x(v_1)} \cdot \frac{\eta}{x(v_3) -
x(v_1)} \\
&=  \frac{\eta}{x(v_2) - x(v_1)} =  \varepsilon.
\end{aligned}
\]

In the third line, we have used the fact that $\Delta'$ is a scaled, translated copy of~$\Delta$.
An analogous calculation shows that $\overline{q_2} - q_2 =
\varepsilon$.

The result is that if $\varepsilon$ is the difference between a current boundary slope $r$ of $G_{i - 1}'$ and the corresponding boundary slope $Z(r)$ of $Z_{h(i- 1)}$, then this difference is propagated, but not increased, by the addition of $a_i'$.  That is, the slopes of the boundary edges of $G_i'$ adjacent to $a_i'$ will differ by $\varepsilon$ from the corresponding boundary slopes of $Z_{\rho(i)}$.

We now investigate how the slopes $\overline{q_1}$ and $\overline{q_2}$ change when we move $\overline{v_3}$ to the integer point $v_3' = a_i'$.  We may assume
without loss of generality that $x(z_{\rho(i)}) < 0$, and hence that $a_i' = (\lfloor x(v_3) \rfloor, \lceil y(v_3) \rceil)$, as the other case is treated
identically.  We will let $q_1'$ and $q_2'$ denote the slopes that result from replacing $\overline{v_3}$ with $v_3' = a_i'$.  That is, let $q_1'$ be the slope
of the line passing through $a_i'$ and $w_1'$, and let $q_2'$ be the slope of the line passing through $a_i'$ and $w_2'$.

Since $z_{\rho(i)}$ lies to the left of the apex of $Z_{\rho(i)}$, the vertex $a_i'$ lies below and to the left of the apex of $G_i'$.  Therefore we clearly have $q_1' - \overline{q_1} > 0$ and $\overline{q_2} - q_2' > 0$.  By $\cP(i - 1, 2)$, we have $|\varepsilon| < i - 1$.  Therefore we obtain
\begin{align*}
& q_1' - \overline{q_1} =  \frac{\lceil y(\overline{v_3}) \rceil - y(w_1')}{\lfloor x(\overline{v_3}) \rfloor - x(w_1')} - \overline{q_1} <
\frac{y(\overline{v_3}) - y(w_1') + 1}{x(\overline{v_3}) - x(w_1') - 1} - \overline{q_1} \\
& = \frac{y(\overline{v_3}) - y(w_1') + 1 - (x(\overline{v_3}) - x(w_1'))\overline{q_1} + \overline{q_1}}{x(\overline{v_3}) - x(w_1') - 1} =  \frac{1 +
\overline{q_1}}{x(\overline{v_3}) - x(w_1') - 1} =  \frac{1 + \overline{q_1}}{x(v_3) - x(v_1) - 1} \\
& \leq  \frac{1 + \overline{q_1}}{x(z_{\rho(i)}) - x(b_1) - 1} \leq  \frac{1 + \overline{q_1}}{\alpha - 1} =  \frac{1 + q_1 + \varepsilon}{\alpha - 1} \leq
\frac{1 + M + \varepsilon}{\alpha - 1} \leq \frac{1 + M + (i - 1)}{\alpha - 1} \\
& \leq \frac{2n^2 + n}{\alpha - 1} = 1.
\end{align*}
In the third line we have used the fact that $x(v_3) - x(v_1) \geq x(z_{\rho(i)}) - x(b_1)$, which we demonstrated above in order to establish $\cP(i, 1)$.  An
analogous calculation shows that $\overline{q_2} - q_2' < 1$.

We may now compute
\[|q_1' - q_1| \leq |q_1' - \overline{q_1}| + |\overline{q_1} - q_1| = |q_1' - \overline{q_1}| + \varepsilon < 1 + \varepsilon \leq 1 + (i - 1) = i.\]
An identical calculation shows that $|q_2' - q_2| \leq i$.  Note that $q_1'$ is the slope of the boundary edge $w_1'a_i'$ of $G_i'$ and $q_1$ is the slope of
the edge $b_1z_{\rho(i)} = Z(w_1'a_i')$, and similarly for $q_2'$ and $q_2$.  This establishes $\cP(i, 2)$.

From the construction of $a_i'$ it is clear that the line segments $w_1'a_i'$ and $a_i'w_2'$ intersect $G_{i - 1}'$ only in the vertices $w_1'$ and~$w_2'$.  Thus $G_i'$ is a plane triangulation, and $G_i' \sim G_i$.  To show that $\F(G_i')$ is projectively convex, we carry out a calculation similar to that of the
$d_i(a_i) > 2$ case.

Let $\hat{q_1}$ denote the slope of the boundary edge of $G_i'$ adjacent and to the left of $w_1'$, if such an edge exists, and let $\hat{q_2}$ denote the slope
of the boundary edge of $G_i'$ adjacent and to the right of $w_2'$, if such an edge exists.  Let $Z(\hat{q_1})$ and $Z(\hat{q_2})$ denote the boundary slopes of
$Z_{h(i - 1)}$ corresponding to $\hat{q_1}$ and $\hat{q_2}$, respectively.  By $\cP(i - 1, 2)$, we have $\overline{q_1} - q_1 = \epsilon \leq i - 1$ and
$Z(\hat{q_1}) - \hat{q_1} \leq i - 1$.  Thus
\[
\begin{aligned}
&\hat{q_1} - \overline{q_1} = (\hat{q_1} - q_1) - (\overline{q_1} - q_1) \geq (\hat{q_1} - q_1) - (i - 1) \\
&= (Z(\hat{q_1}) - q_1) - (Z(\hat{q_1}) - \hat{q_1}) - (i - 1) \geq (Z(\hat{q_1}) - q_1) - 2(i - 1) \\
&\geq 2n - 2i + 2 \geq 2.
\end{aligned}
\]
We conclude that
\[\hat{q_1} - q_1' = (\hat{q_1} - \overline{q_1}) - (q_1' - \overline{q_1}) \geq (\hat{q_1} - \overline{q_1}) - 1 \geq 2 - 1 = 1 > 0.\]
An analogous calculation shows that $q_2' - \hat{q_2} > 0$.  Thus $w_1'$ and $w_2'$ are convex vertices of $G_i'$.  Because the region $\F(G_{i - 1}')$ is
projectively convex by $\cP(i - 1, 3)$, we conclude that $\F(G_i')$ is projectively convex.  By $\cP(i - 1, 3)$, the sequence $(a_1', \ldots, a_{i - 1}')$ is a
convex shedding sequence for $G_{i - 1}'$, hence $(a_1', \ldots, a_i')$ is a convex shedding sequence for $G_i'$.  Thus $\cP(i, 3)$ holds.  We have now
established $\cP(i)$ in the case that $d_i(a_i) = 2$.  This completes the induction, and we conclude that $\cP(i)$ holds for all $3 \leq i \leq n$.  Thus the
triangulation $G' = G_n'$ is a sequentially convex embedding of $G$, with convex shedding sequence $\ba' = (a_1', \ldots, a_n')$.

We have immediately that the $x$ dimension of $G'$ is \[\alpha2(n - 2) = (2n^2 + n + 1)(2n - 4) =  4 n^3 - 6 n^2 - 2n - 4 \leq 4n^3.\]  Since $\cP(n, 2)$ holds,
we conclude that the largest absolute value of a boundary slope of $G'$ is at most $M + n \leq 2n^2 + n$.  Thus the $y$ dimension of $G'$ is at most \[\alpha2(n
- 2)(2n^2 + n) = (4 n^3 - 6 n^2 - 2n - 4)(2n^2 + n) = 8n^5 - 8n^4 - 10n^3 - 10n^2 - 4n \leq 8n^5.\]  Therefore $G'$ is embedded in a $4n^3 \times 8n^5$ integer
grid.  \end{proof}

\bigskip

\section{The shedding diameter}\label{s:diam}

\noindent
Let $G = (V, E)$ be a plane triangulation and let $\A_G$ denote the set of all shedding sequences for $G$.  For $\ba = (a_1, \ldots, a_n) \in \A_G$, we write
$a_j \rightarrow_\ba a_i$ if $a_j$ is adjacent to $a_i$ in~$G_i(\ba)$.  Then we define the \emph{height} of each vertex $a_i$ recursively, by
\begin{equation*}
\tau(a_i, \ba) = \begin{cases}
i & i \leq 3 \\
1 + \max\{\tau(a_j, \ba) \ | \ a_j \rightarrow_\ba a_i\} & i > 3
\end{cases}.
\end{equation*}
We define the \emph{height} of the shedding sequence $\ba \in \A_G$ by \[\tau(\ba) = \max_i\tau(a_i, \ba),\] and the \emph{shedding diameter} of $G$ by \[\tau(G)
= \min_{\ba \in \A_G} \tau(\ba).\]

Taking the transitive closure of the relation $\rightarrow_\ba$, we obtain a partial order $\preceq_\ba$ on the vertices of $G$.  The height $\tau(\ba)$ of the
sequence $\ba$ is then precisely the height of $\preceq_\ba$.  That is, $\tau(\ba)$ is the maximal length of a chain in $\preceq_\ba$.

The next lemma involves the following intuitive notion.  Let $\pi : \R^3 \rightarrow \R^2$ denote the coordinate projection $\pi(x, y, z) = (x, y)$.  We say that
a convex polyhedron $P \subset \R^3$ (possibly unbounded) \emph{projects vertically} onto a geometric plane graph $G$, if $\pi(P) = \F(G)$, and $\pi$ induces an
isomorphism on the face structures of $P$ and $G$.  This last condition means that $w_1, \ldots, w_k$ are the vertices of a facet ($2$-face) of $P$ if and only
if $\pi(w_1), \ldots, \pi(w_k)$ are the vertices of a face of $G$.

\begin{lemma}  Let $G$ be a plane triangulation with $n$ vertices and shedding sequence $\ba \in \A_G$, embedded as in Theorem \ref{plane}, so that $\ba$ is a
convex shedding sequence for $G$.  Then there is a convex polyhedron $P_i$ that projects vertically onto $G_i$, for each $i = 3, \ldots, n$.  Furthermore, if
$h(a_i)$ denotes the height of the vertex of $P_i$ projecting to $a_i$, then we may choose $h(a_i)$ to be an integer such that $h(a_i) \leq 499 n^8 m_i + 1$,
where \[m_i~=~\max\{h(a_j) \ | \ a_j \rightarrow_\ba a_i\}.\] \label{height} \end{lemma}

\begin{proof}  We proceed by induction on $i$.  Let $h(v)$ denote the height assigned to the vertex $v \in V(G)$, and let $\varphi(v) = (x(v), y(v), h(v)) \in
\R^3$ denote the point of $\R^3$ projecting vertically to $v$.  We define $h(a_1) = h(a_2) = h(a_3) = 0$, and let \[P_3 = \{(x, y, z) \in \R^3 \ | \ (x, y) \in
\conv(a_1, a_2, a_3), z \geq 0\}.\]  That is, $P_3$ is the unbounded prism with triangular face $a_1a_2a_3$ and lateral edges extending in the positive vertical
direction, parallel to the $z$-axis.

If $i > 3$, then by the induction hypothesis, there is a convex polyhedron $P_{i - 1}$ that projects vertically onto $G_{i - 1}$.  So in particular, the vertices
of $P_{i - 1}$ are $\varphi(a_1), \varphi(a_2), \ldots, \varphi(a_{i - 1})$.  To obtain a lifting of $G_i$, we must choose $h(a_i)$ properly.  Namely, we must
choose $h(a_i)$ large enough to ensure that $\varphi(a_i)$ is in convex position with respect to $\varphi(a_1), \varphi(a_2), \ldots, \varphi(a_{i - 1})$.

Let $\Sc_i$ denote the set of faces of $G_{i - 1}$ having a vertex $v$ such that $v \rightarrow_\ba a_i$, and let $\varphi(\Sc_i)$ denote the facets of $P_{i - 1}$
that project vertically to the faces of $\Sc_i$.  We choose the height $h(a_i)$ large enough so that for every facet $F \in \varphi(\Sc_i)$, the point
$\varphi(a_i)$ lies \emph{above} the hyperplane spanned by $F$.  That is, we require that $\varphi(a_i) - (0, 0, k)$ is coplanar with $F$ for some $k > 0$.

Let $\ell_i$ denote the ray with vertex $\varphi(a_i)$ and extending in the positive vertical direction, parallel to the $z$-axis.  Then we define a convex
polyhedron $P_i = \conv(P_{i - 1} \cup \ell_i)$.  By the choice of $h(a_i)$, the point $\varphi(a_i)$ lies above all facet hyperplanes of $\varphi(\Sc_i)$, hence
above all facet hyperplanes of $P_{i - 1}$.  Thus the vertices of $P_i$ are $\varphi(a_1), \varphi(a_2), \ldots \varphi(a_i)$, and $\varphi(a_i)$ is not a vertex
of any facet of $P_{i - 1}$.  This last fact implies, because $G_i$ is a triangulation, that all faces in $\cF(G_i) \sm \cF(G_{i - 1})$ are obtained from the
projection of the new facets of $P_i$.  On the other hand, because $\F(G_i)$ is convex, all new facets of $P_i$ project vertically to faces in $\cF(G_i) \sm
\cF(G_{i - 1})$.   Since $P_{i - 1}$ projects vertically onto $G_{i - 1}$, these last two statements imply that $P_i$ projects vertically onto $G_i$.

Now we determine an upper bound on the height $h(a_i)$ necessary for the above construction.  To do this, we determine an upper bound on the coordinate $z$ for
which $(x(a_i), y(a_i), z)$ is coplanar with some facet in $\varphi(\Sc_i)$.  If we take $h(a_i)$ to be any integer greater than this upper bound, then $\varphi(a_i)$
lies above the hyperplane of every facet in $\varphi(\Sc_i)$.

We write $x_0 = x(a_i)$, $y_0 = y(a_i)$, and let $z_0 > 0$.  Fix $F \in \Sc_i$ and let $v_1, v_2, v_3$ denote the vertices of $F$. Let $(x_j, y_j, z_j) \in \R^3$
denote the coordinates of $\varphi(v_j)$.  So in particular $(x_0, y_0) = a_i$, and $(x_j, y_j) = v_j$ for $j = 1, 2, 3$.  Suppose that $(x_0, y_0, z_0)$ is
coplanar with $\varphi(v_1), \varphi(v_2), \varphi(v_3)$.  This means that \[z_0 = c_1x_0 + c_2y_0 + c_3,\] where $c_1, c_2, c_3$ satisfy
\[
\left(\begin{array}{ccc}
	x_1 & y_1 & 1 \\
	x_2 & y_2 & 1 \\
	x_3 & y_3 & 1
	\end{array}\right)
	\left(\begin{array}{c}
	c_1 \\
	c_2 \\
	c_3
	\end{array}\right)
	=
	\left(\begin{array}{c}
	z_1 \\
	z_2 \\
	z_3
	\end{array}\right).
\]

Let $A$ denote the matrix on the left side of this equation, and write \[\bx = (x_1, x_2, x_3), \quad \quad \by = (y_1, y_2, y_3), \quad \quad \bz = (z_1, z_2,
z_3).\]  By Cramer's rule, $c_i = \frac{\det(A_i)}{\det(A)}$, where $A_i$ is obtained by replacing the $i^{\text{th}}$ column of $A$ with~$\bz^T$.

Since $G$ is embedded as in Theorem~\ref{plane}, the vertices of $G$ lie in a $4n^3 \times 8n^5$ integer grid.  Furthermore, from the construction of
Theorem~\ref{plane}, the point $(0, 0)$ is contained in the edge $a_1a_2$ of $G$.  This implies that $|x_j| \leq 4n^3$ and $|y_j| \leq 8n^5$ for each $j = 0, 1,
2, 3$.  Therefore \[\|\bx\| \leq \sqrt{3}\max_{1 \leq j \leq 3}|x_j| \leq 4\sqrt{3}n^3 \quad \quad \text{and} \quad \quad \|\by\| \leq \sqrt{3}\max_{1 \leq j
\leq 3}|y_j| \leq 8\sqrt{3}n^5.\]  Note also that
\[\|\bz\| \leq \sqrt{3}\max_{1 \leq j \leq 3}|z_i| \leq \sqrt{3}m_i.\]

Since $A$ is an invertible integer matrix, we have $|\det(A)| \geq 1$.  Thus
\[
\begin{aligned}
& \frac{|\det(A_1)|}{|\det(A)|} \leq  |\det(A_1)| \\
&\leq  \|(1, 1, 1)\|\|\bz\|\|\by\| =  \sqrt{3}\|\bz\|\|\by\| \\
&\leq  \sqrt{3}(\sqrt{3} m_i)(8\sqrt{3}n^5) =  24\sqrt{3} n^5 m_i,
\end{aligned}
\]
where we have used Hadamard's inequality in the second line.

By a similar argument,
\[\frac{|\det(A_2)|}{|\det(A)|} \leq 12\sqrt{3} n^3 m_i \quad \quad \text{and} \quad \quad \frac{|\det(A_3)|}{|\det(A)|} \leq 96\sqrt{3}n^{8} m_i.\]

Thus when $(x_0, y_0, z_0)$ is coplanar with $\varphi(v_1), \varphi(v_2), \varphi(v_3)$, we have
\begin{align*}
& z_0 =  \dfrac{\det(A_1)}{\det(A)}x_0 + \dfrac{\det(A_2)}{\det(A)}y_0 + \dfrac{\det(A_3)}{\det(A)} \\
& \leq \dfrac{|\det(A_1)|}{|\det(A)|}|x_0| + \dfrac{|\det(A_2)|}{|\det(A)|}|y_0| + \dfrac{|\det(A_3)|}{|\det(A)|} \\
& \leq 24\sqrt{3} n^5 m_i(4n^3) + 12\sqrt{3} n^3 m_i (8n^5) + 96\sqrt{3}n^{8} m_i \\
&= 288\sqrt{3} n^{8}m_i \leq  499 n^8 m_i.
\end{align*}

So letting $z_0$ be the smallest integer greater than $499n^8 m_i$ will ensure that $(x_0, y_0, z_0)$ lies above the hyperplane containing $\varphi(v_1),
\varphi(v_2), \varphi(v_3)$.  Thus we may take $h(a_i) \leq 499n^8 m_i + 1$, as desired. \end{proof}

\begin{thm}  Let $G$ be a plane triangulation with $n$ vertices.  Then $G$ is the vertical projection of a convex $3$-polyhedron with vertices lying in a $4n^3
\times 8n^5 \times (500n^8)^{\tau(G)}$ integer grid. \label{t:size} \end{thm}

\begin{proof}  Choose a shedding sequence $\ba \in \A_G$ such that $\tau(G) = \tau(\ba)$.  By Theorem \ref{plane}, we may embed $G$ in a $4n^3 \times 8n^5$
integer grid such that $\ba = (a_1, \ldots, a_n)$ is a convex shedding sequence for $G$.  For each vertex $a_i$ we may assign a height $h(a_i)$ as follows.  For
$i = 1, 2, 3$ we may set $z_i = 0$.  For $i > 3$, by Lemma \ref{height} we may choose $h(a_i)$ such that $G_i$ is the projection of a polyhedral surface, and
\[h(a_i) \leq \left(499n^8 + 1\right)^{\tau(a_i, \ba)} \leq (500n^8)^{\tau(a_i, \ba)} \leq (500n^8)^{\tau(\ba)} = (500n^8)^{\tau(G)}.\] \end{proof}

Note that if the boundary of $G$ is a triangle (that is, $\partial \F(G)$ contains exactly three vertices), then the polyhedron of Theorem~\ref{t:size} may be
replaced with a (bounded) $3$-polytope.  Indeed, simply truncate the polyhedron with the hyperplane that is defined by the lifts of the three boundary vertices
of $G$.  Then the three boundary vertices of $G$ lift to the vertices of a triangular face of the resulting $3$-polytope.

\bigskip

\section{Triangulations of a rectangular grid}\label{s:grid-tri}

\noindent
For $p, q \in \Z$, $p, q \geq 2$, let $[p \times q] = \{1, \ldots, p\} \times \{1, \ldots, q\}$.  We may think of the integer lattice $[p \times q]$ as the
vertices of $(p - 1)(q - 1)$ unit squares.  A geometric plane triangulation $G$ is a \emph{triangulation of $[p \times q]$} if the vertices of $G$ are exactly
the vertices of $[p \times q]$, and every boundary edge of $[p \times q]$ is an edge of $G$.  We call $G$ a \emph{grid triangulation}.  An $\ell \times \ell$
\emph{subgrid} of $\Z^2$ is an integer translation of the lattice $[\ell \times \ell] = \{1, \ldots, \ell\} \times \{1, \ldots, \ell\}$.  By an $\ell \times
\ell$ subgrid of $[p \times q]$ we mean an $\ell \times \ell$ subgrid of $\Z^2$ that is a subset of $[p \times q]$.  In this section we state and prove the
following result concerning the shedding diameter of grid triangulations.

\begin{thm}  Let $G$ be a triangulation of $[p \times q]$ such that for every edge $e$ of $G$, there is a subgrid of size $\ell \times \ell$ that contains the
endpoints of $e$.  Then $\tau(G) \leq 6\ell(p + q)$. \label{t:diam} \end{thm}

This gives a class of triangulations with sublinear shedding diameter, if $\ell$ is held constant.  According to Theorem~\ref{t:size}, such a triangulation can
be drawn in the plane so that it is the vertical projection of a simplicial $3$-polyhedron embedded in a subexponential grid.  That is, this class of
triangulations corresponds to a class of simplicial polyhedra which may be embedded in an integer grid whose size is subexponential in the number of vertices.

Let $\leq_{\Z^2}$ denote the linear order on $\Z^2$ defined by
\[\text{$(x_1, y_1) \leq_{\Z^2} (x_2, y_2)$ if and only if $y_1 < y_2$ or $y_1 = y_2$ and $x_1 \leq x_2$.}\]
That is, $\leq_{\Z^2}$ is a lexicographic order in which $y$-coordinates take precedence in determining the order.  We state without proof the following lemma,
which summarizes some standard properties of shedding vertices of planar triangulations (see~\cite[$\S3$]{BP} for a proof and references).

\begin{lemma} [\cite{BP}]  Let $G$ be a plane triangulation, and let $v$ be a boundary vertex of $G$.  Then either~$v$ is a shedding vertex of $G$, or is the endpoint of a diagonal $e$ of~$G$.  Furthermore, each of the two connected components of $\F(G) \sm e$ contains a shedding vertex of $G$.
\label{l:diag} \end{lemma}

The rough idea of the proof of Theorem~\ref{t:diam} is as follows (we provide the details below).  We begin by constructing a particular shedding sequence $\ba$
for $G$.  To do this, we first subdivide $[p \times q]$ into a grid of $\lceil \frac{pq}{\ell^2} \rceil$ subgrids, (most of) which are squares of size $\ell
\times \ell$.  These squares form $\lceil \frac{p}{\ell} \rceil$ columns and $\lceil \frac{q}{\ell} \rceil$ rows.

We shed $G$ in three stages.  In Stage 1, we take every fourth column $U(1), U(5), U(9), \ldots$ and shed the vertices of each of these columns from top to
bottom.  When shedding $U(i)$, we may need to shed vertices in the column $U(i - 1)$ or $U(i + 1)$, for a total of at most $3q\ell$ vertices shed in the process
of shedding the column $U(i)$.  Because of their spacing, the shedding vertices in each column do not interact.  Specifically, at each step we have a collection
of shedding vertices, one from each column, which we may think of as shedding ``all at once".  This collection of vertices is then an antichain with respect to
$\preceq_\ba$.  When shedding the vertices of each such column, for topological reasons we do not shed the vertices $(x, y)$ with $y \leq \ell$.  See
Figure~\ref{f:gridshed}.

After Stage 1 is complete, what remains are a set of ``jagged tricolumns", each of which consists of the remaining vertices of three adjacent columns.  Hence each
jagged tricolumn contains at most $3q\ell$ vertices.  In Stage 2, we shed these columns, but for topological reasons we do not shed vertices $(x, y)$ with $y \leq 2
\ell$.  As before, these jagged tricolumns do not interact, and at each step we have a set of shedding vertices, each of which belongs to a different jagged tricolumn.
Hence this set forms an antichain.  Finally, in Stage 3 we shed the remaining vertices, which are contained in the bottom two rows of $G$.  There are at most
$2p\ell$ such vertices, and we simply define a singleton antichain for each of them.  Therefore we see that $G$ may be partitioned into at most $2p\ell + 3q\ell
+ 3q\ell = \ell(2p + 6q) \leq 6\ell(p + q)$ antichains of~$\preceq_\ba$.  This implies that $\tau(\ba) \leq 6\ell(p + q)$, since $\tau(\ba)$ is the length of
some chain in $\preceq_\ba$.  The detailed proof follows.

\begin{figure}[t!]
 \begin{center}
     \begin{overpic}[height=5.5cm]{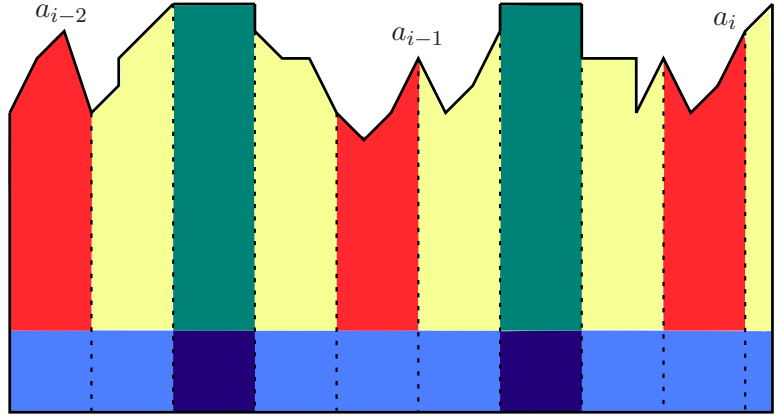}
     \put(3.5, 52){$a_{i - 2}$}
     \put(50, 49){$a_{i - 1}$}
     \put(92, 51){$a_i$}
      \end{overpic}
   \caption{A graph $G_i$ produced during Stage 1 of the construction of $\ba$ in the proof of Theorem~\ref{t:diam}.  Distinct columns $U(j)$ are separated by
   dashed lines.  The columns of the form $U(1 + 4j)$ are shown in red, while the columns $U(3 + 4j)$ are shown in green.  The bottom row $R(1)$ is shown in
   blue.}
   \label{f:gridshed}
 \end{center}
\end{figure}

\begin{proof}[Proof of Theorem~\ref{t:diam}]  Let $G$ be such a triangulation of $[p \times q]$.  For $i \in \Z$, let
\begin{align*}
U(i) & = \{(x, y) \in  [p \times q] \ | \ \ell(i - 1) + 1 \leq x \leq \ell i\}, \quad \text{and} \\
R(i) & = \{(x, y) \in  [p \times q] \ | \ \ell(i - 1) + 1 \leq y \leq \ell i\}.
\end{align*}
Many of these sets are empty (for example when $i \leq 0$).  We think of the sets $U(i)$ as \emph{columns} of width $\ell$ and the sets $R(i)$ as \emph{rows} of height $\ell$.  For each
$i \in \Z$, we also define
\[T(i) = U(i - 1) \cup U(i) \cup U(i + 1),\]
which may be empty.  We call $T(i)$ a \emph{tricolumn} of $[p \times q]$.

We construct the shedding sequence $\ba$ recursively.  Suppose that we have a sequence of shedding vertices $a_{i + 1}, a_{i + 2}, \ldots, a_n$ (for the initial
step of the recursion, $i = n$ and this sequence is empty), and therefore we also have plane triangulations $G_i, G_{i + 1}, \ldots, G_n = G$, where as usual $G_{j - 1} = G_j
- \{a_j\}$ for all $j = i + 1, \ldots, n$.  For each $i = 1, \ldots, n$, let $R_i(1)$ denote the subgraph of $G_i$ induced by the vertices in $R(1)$.  Similarly, for each $i = 1, \ldots, n$ and $j \in \Z$, let $U_i(j)$ denote the subgraph of $G_i$ induced by the vertices in $U(j)$.
We let $\cP(i, 1)$, $\cP(i, 2)$, and $\cP(i, 3)$ denote the following statements:
\begin{align*}
\cP(i, 1). & \quad \text{$U_i(3 + 4j) = U_n(3 + 4j)$ (for all $j \in \Z$ such that $U(3 + 4j) \neq \emptyset$).} \\
\cP(i, 2). & \quad \text{$R_i(1)$ is connected.} \\
\cP(i, 3). & \quad \text{$\{1, 2, \ldots, p\} \times \{1\} \subset V(G_i)$.}
\end{align*}
Note that $\cP(n, 1)$ holds trivially.  Furthermore, we have $G_n = G$, so the vertices of $\nobreak{U_n(1 + 3j)}$ are exactly those of $U(1 + 3j)$, and
similarly for $R_n(1)$ and $R(1)$.  Since $G$ is a grid triangulation, it follows that $U_n(1 + 3j)$ and $R_n(1)$ are connected.  In particular, $\cP(n, 2)$
holds.  Finally, $\cP(i, 3)$ says that $G_i$ contains all vertices $v$ of $[p \times q]$ for which $y(v) = 1$.  Clearly $\cP(n, 3)$ holds.

To construct the next vertex $a_i$ of the shedding sequence, we break the construction into three stages, described below.  As can readily be seen, each stage occurs for a consecutive sequence of indices.  That is, there are integers $i_2 < i_1$ such that Stage~1 occurs
for $\nobreak{i = i_1, i_1 + 1, \ldots, n}$, Stage 2 occurs for $\nobreak{i = i_2, i_2 + 1, \ldots, i_1 - 1}$, and
Stage~3 occurs for $\nobreak{i = 1, 2, \ldots, \nobreak{i_2 - 1}}$.  We will also show, as we describe these stages, that $\cP(i, 1)$ and $\cP(i, 3)$ hold for $i = i_1, i_1 + 1, \ldots, n$, and $\cP(i, 2)$ holds for $i = i_2, i_2 + 1, \ldots, \nobreak{i_1 - 1}$.  That is, $\cP(i, 1)$ and $\cP(i, 3)$ hold through all of Stage 1, and $\cP(i, 2)$ holds through all of Stage~2.  We mention that $\cP(i, 3)$ holds through all of Stage 2 as well, but we will not need this fact.

\medskip

\noindent \textbf{Stage 1.}  Some column of the form $U(1 + 4j)$ contains a vertex $(x, y)$ of $G_i$ with $y > \ell$.  See Figure~\ref{f:gridshed}.  Let $U(1 +
4j_1), \ldots, U(1 + 4j_r)$ denote all such columns, where $j_1 < \cdots < j_r$.  Assume that $\cP(i, 1)$ and $\cP(i, 3)$ hold.

For each $k = 1, \ldots, r$, let $v_k$ be the $\leq_{\Z^2}$-greatest vertex of $U(1 + 4j_k)$.  If $v_k$ is a shedding vertex of $G_i$, define $w_k = v_k$.
Otherwise, by Lemma~\ref{l:diag}, the vertex $v_k$ is the endpoint of a diagonal of $G_i$.  Let $u_k$ denote $\leq_{\Z^2}$-greatest vertex of $G_i$ such that the
edge $u_kv_k$ is a diagonal of $G_i$.  Write $e_k = u_kv_k$.

By the Jordan curve theorem, $\F(G_i) \sm e_k$ has two connected components, call them $A_k$ and $A_k'$.  Since the vertices $u_k$ and $v_k$ are
adjacent, by assumption they are contained in an $\ell \times \ell$ subgrid of $[p \times q]$.  It follows that $u_k \in T(1 + 4j_k)$.  Thus $u_k, v_k \notin U(3
+ 4j)$ for all~$j$.  Furthermore, since $y(v_k) > \ell$ we have $y(u_k) > 1$.  Then by $\cP(i, 1)$ and $\cP(i, 3)$, one of the components of $\F(G_i) \sm e_k$, say $A_k$, does not intersect any of the columns $U(3 + 4j)$, and does not contain any vertices $v$ with $y(v) = 1$.  For otherwise, we could connect the components of $\F(G_i) \sm e_k$ with a path from a vertex in $\{1, 2, \ldots, p\} \times \{1\}$ to a vertex in a column $U(3 + 4j)$.  It follows that all vertices in $A_k$ are contained in $T(1 + 4j_k)$.  By Lemma~\ref{l:diag}, the region $A_k$
contains a shedding vertex of $G_i$.  We define $w_k$ to be the $\leq_{\Z^2}$-greatest such shedding vertex.

We now have a collection of shedding vertices $w_1, \ldots, w_r$ of $G_i$.  Clearly the neighbors of each vertex $w_k$ lie in the tricolumn $T(1 + 4j_k)$, so no two of the vertices $w_1, \ldots w_k$ are adjacent to a common vertex.  Thus the vertex $w_{r - 1}$ is a shedding vertex of $G_i - \{w_r\}$, the vertex $w_{r - 2}$ is a shedding vertex of
$G_i - \{w_r, w_{r - 1}\}$, etc.  That is, these vertices remain shedding vertices after deleting any finite subset of them from $G_i$.  So for each $k = 1,
\ldots, r$, we may define $a_{i - r + 1}, \ldots, a_i$ by $a_{i - r + k} = w_k$.  Since no two of the vertices $a_{i - r + 1}, \ldots a_i$, are adjacent, the set
$\{a_{i - r + 1}, \ldots, a_i\}$ is an antichain of $\preceq_\ba$.

We will write $i(k) = i - r + k$.  We now show inductively that $\cP(i(k) - 1, 1)$ and $\cP(i(k) - 1, 3)$ hold, for all $k = 1, \ldots, r$.  From the above definition of the shedding vertices $a_{i(k)}$, we have $a_{i(k)} \notin U(3 + 4j)$ for all $j$.  Thus $U_{i(k) - 1}(3 + 4j) = U_i(3 + 4j) = U_n(3 + 4j)$, where the last equality follows from $\cP(i, 1)$.  That is, $\cP(i(k) - 1, 1)$ holds for all $k = 1, \ldots, r$.  Similarly, we see that $y(a_{i(k)}) > 1$, so the vertices of $G_{i(k) - 1}$ with $y$-coordinate $1$ are the same as those of $G_i$ with $y$-coordinate $1$.  It follows from $\cP(i, 3)$ that $\cP(i(k) - 1, 3)$ holds for all $k = 1, \ldots, r$.

This completes the description of Stage 1.  Before describing Stage 2, we show that $\cP(i_1 - 1, 2)$ holds, where $i_1 \in \{1, 2, \ldots, n\}$ is the least index for which Stage 1 occurs.  That is, we wish to show that at the beginning of Stage 2 (when $i = i_1 - 1$), the graph $R_{i_1 - 1}(1)$ is connected.  To this end, we introduce the following notation.  Let $R_i(1, j)$ denote the subgraph of $G_i$ induced by the vertices $T(1 + 4 j) \cap R(1)$, for $i = 1, \ldots, n$ and $j \in \Z$.

Note that in Stage 1, since $\cP(i, 3)$ holds for all $i = i_1, \ldots, n$, the induced subgraph of $G_i$ on the vertices $\{1, \ldots, p\} \times \{1\}$ is a (connected) path in $R_i(1)$, which clearly intersects every column $U(j)$ of $G_i$.  Therefore, because $\cP(i, 1)$ also holds in Stage 1, to show that $R_{i_1 - 1}(1)$ is connected, we only need to show that for each fixed $j$, the graph $R_{i_1 - 1}(1, j)$ is connected.

Let $i \in \{i_1, i_1 + 1, \ldots, n\}$, and fix a column $U(1 + 4 j_k)$.  Note that $\cP(n, 2)$ holds, so in particular $R_n(1, j_k)$ is connected.  We assume that $R_{i(k)}(1, j_k)$ is connected, and we will show that there is an index $i_1 \leq i' \leq i(k)$ such that $R_{i' - 1}(1, j_k)$ is connected.  If $a_{i(k)} \notin R(1)$, then $R_{i(k) - 1}(1, j_k) = R_{i(k)}(1, j_k)$.  Therefore we may simply take $i' = i(k)$ in this case.

Now suppose that $a_{i(k)} \in R(1)$.  Note that $y(v_k) > \ell$, and therefore $v_k \notin R(1)$.  So $w_k = a_{i(k)} \neq v_k$, and thus $w_k \in A_k$, where $A_k$ is the component of $\F(G_i) \sm e_k$ defined above.  Consider a path $\gamma$ in $R_{i(k)}(1, j_k)$ whose endpoints are not in $A_k$, and which passes through $a_i$.  Then $\gamma$ must both enter and exit the component $A_k$ through the vertex $u_k$ of $e_k$, because $v_k \notin R(1)$.  Thus $\gamma$ can be replaced with a path $\gamma'$ in $R_{i(k)}(1, j_k)$ having the same endpoints, such that $\gamma'$ contains no vertices of $A_k$.  Then from the assumption that $R_{i(k)}(1, j_k)$ is connected, we conclude that
\begin{equation}
\text{$R_{i(k)}(1, j_k) \sm A_k$ is connected.}
\label{e:stage1}
\end{equation}

From the above construction of the vertices $w_k$, we see that in later steps $s$ of Stage~1, we will always define the shedding vertex $w$ for tricolumn $T(1 + 4 j_k)$ to be a vertex such that $w \in A_k \cap G_{s}$, until the set $A_k \cap G_s$ is empty.  So let $i'$ denote the step at which $a_{i'}$ is the last remaining vertex of $A_k \cap G_{i'}$.  Then $R_{i' - 1}(1, j_k) = R_{i(k)}(1, j_k) \sm A_k$, and therefore~(\ref{e:stage1}) implies that $R_{i' - 1}(1, j_k)$ is connected.

It follows that $R_{i_1 - 1}(1)$ is connected.  That is, $\cP(i_1 - 1, 2)$ holds.

\medskip

\noindent \textbf{Stage 2.}  No column of the form $U(1 + 4j)$ contains vertices $(x, y)$ of $G_i$ with $y > \ell$, but some tricolumn of the form $T(3 + 4j)$
contains vertices $(x, y)$ of $G_i$ with $y > 2 \ell$.  Let $T(3 + 4j_1), \ldots, T(3 + 4j_r)$ denote all such tricolumns, where $j_1 < \cdots < j_r$.  Assume that
$\cP(i, 2)$ holds.

For each $k = 1, \ldots, r$, let $v_k$ be the $\leq_{\Z^2}$-greatest vertex of $T(3 + 4j_k)$.  If $v_k$ is a shedding vertex of $G_i$, define $w_k = v_k$.
Otherwise, by Lemma~\ref{l:diag}, the vertex $v_k$ is the endpoint of a diagonal of $G_i$.  Let $u_k$ denote the $\leq_{\Z^2}$-greatest vertex of $G_i$ such that the
edge $u_kv_k$ is a diagonal of $G_i$.  Write $e_k = u_kv_k$.

By the Jordan curve theorem, $\F(G_i) \sm e_k$ has two connected components, call them $A_k$ and $A_k'$.  Since the vertices $u_k$ and $v_k$ are
adjacent, by assumption they are contained in an $\ell \times \ell$ subgrid of $[p \times q]$.  Since $y(v_k) > 2\ell$, it follows that $y(u_k) > \ell$, and thus $u_k, v_k \notin R(1)$.  Then by $\cP(i, 2)$, one of the components of $\F(G_i) \sm e_k$, say $A_k$, does not intersect $R(1)$.  By definition of Stage 2, we have 
\begin{equation}
V(U_i(1 + 4j)) \subseteq R(1), \quad j \in \Z,
\label{e:stage2}
\end{equation} 
so we also conclude that $A_k$ does not intersect any column of the form $U(1 + 4j)$.  By Lemma~\ref{l:diag}, the
region $A_k$ contains a shedding vertex of $G_i$.  We define $w_k$ to be the $\leq_{\Z^2}$-greatest such shedding vertex.  Note that $w_k \in T(3 + 4j_k)$ in this case as well, for otherwise, either $A_k$ contains a vertex in $U(1 + 4j_k)$ or $U(5 + 4j_k)$, or $G_i$ has an edge $uv$ with $|x(u) - x(v)| > \ell$.

We now have a collection of shedding vertices $w_1, \ldots, w_r$ of $G_i$.  Every vertex $w_k$ lies in the tricolumn $T(3 + 4j_k)$, and none of the neighbors of $w_k$ are contained in $R(1)$.  This implies, by~(\ref{e:stage2}), that no two of the vertices $w_1, \ldots, w_r$ are adjacent to a common vertex.  Thus these
vertices remain shedding vertices after deleting any finite subset of them from~$G_i$.  So for each $k = 1, \ldots, r$, we may define $a_{i - r + 1}, \ldots,
a_i$ by $a_{i - r + k} = w_k$.  Since no two of the vertices $a_{i - r + 1}, \ldots a_i$, are adjacent, the set $\{a_{i - r + 1}, \ldots, a_i\}$ is an antichain
of $\preceq_\ba$.

Finally, note that by construction we have $a_{i(k)} \notin R(1)$ for all $k = 1, \ldots, r$.  That is, none of the vertices of the row $R(1)$
are deleted in Stage 2.  Thus $R_{i(k) - 1}(1) = R_i(1)$ for all $k = 1, \ldots, r$, so from $\cP(i, 2)$ we conclude that $\cP(i(k) - 1, 2)$ holds for all $k = 1, \ldots, r$.

\medskip

\noindent \textbf{Stage 3.}  All vertices $(x, y)$ of $G_i$ have $y \leq 2\ell$.  If $i > 3$ we define $a_i$ to be the $\leq_{\Z^2}$-greatest shedding vertex of
$G_i$, which exists by Lemma~\ref{shedding}.  If $i \leq 3$ we define $a_i$ to be the $\leq_{\Z^2}$-greatest vertex of~$G_i$.  Clearly, the singleton set
$\{a_i\}$ is an antichain of~$\preceq_\ba$.

\begin{figure}[t!]
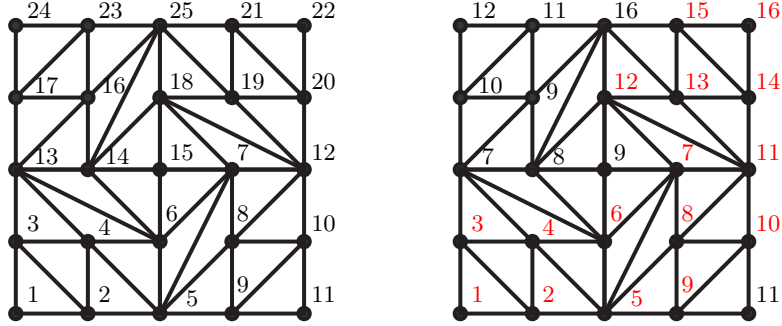

 \begin{center}
 \footnotesize
     \begin{overpic}[height=4.5cm]{grid1}
     \put(11, 95){\textcolor{black}{$24$}}
     \put(32, 95){\textcolor{black}{$23$}}
     \put(53, 95){\textcolor{black}{$25$}}
     \put(74, 95){\textcolor{black}{$21$}}
     \put(95, 95){\textcolor{black}{$22$}}

     \put(13, 73){\textcolor{black}{$17$}}
     \put(33, 73){\textcolor{black}{$16$}}
     \put(53, 74){\textcolor{black}{$18$}}
     \put(74, 74){\textcolor{black}{$19$}}
     \put(95, 74){\textcolor{black}{$20$}}

     \put(13, 52){\textcolor{black}{$13$}}
     \put(34, 52){\textcolor{black}{$14$}}
     \put(53, 53){\textcolor{black}{$15$}}
     \put(73, 53){\textcolor{black}{$7$}}
     \put(95, 53){\textcolor{black}{$12$}}

     \put(11, 32){\textcolor{black}{$3$}}
     \put(32, 31){\textcolor{black}{$4$}}
     \put(52, 35){\textcolor{black}{$6$}}
     \put(73, 35){\textcolor{black}{$8$}}
     \put(95, 32){\textcolor{black}{$10$}}

     \put(11, 11){\textcolor{black}{$1$}}
     \put(32, 11){\textcolor{black}{$2$}}
     \put(58, 10){\textcolor{black}{$5$}}
     \put(73, 14){\textcolor{black}{$9$}}
     \put(95, 11){\textcolor{black}{$11$}}
     \end{overpic}
      \quad \quad \quad \quad
      \begin{overpic}[height=4.5cm]{grid1}
      \put(11, 95){\textcolor{black}{$12$}}
     \put(32, 95){\textcolor{black}{$11$}}
     \put(53, 95){\textcolor{black}{$16$}}
     \put(74, 95){\textcolor{red}{$15$}}
     \put(95, 95){\textcolor{red}{$16$}}

     \put(13, 73){\textcolor{black}{$10$}}
     \put(33, 71){\textcolor{black}{$9$}}
     \put(53, 74){\textcolor{red}{$12$}}
     \put(74, 74){\textcolor{red}{$13$}}
     \put(95, 74){\textcolor{red}{$14$}}

     \put(14, 52){\textcolor{black}{$7$}}
     \put(35, 52){\textcolor{black}{$8$}}
     \put(53, 53){\textcolor{black}{$9$}}
     \put(73, 53){\textcolor{red}{$7$}}
     \put(95, 53){\textcolor{red}{$11$}}

     \put(11, 32){\textcolor{red}{$3$}}
     \put(32, 31){\textcolor{red}{$4$}}
     \put(52, 35){\textcolor{red}{$6$}}
     \put(73, 35){\textcolor{red}{$8$}}
     \put(95, 32){\textcolor{red}{$10$}}

     \put(11, 11){\textcolor{red}{$1$}}
     \put(32, 11){\textcolor{red}{$2$}}
     \put(58, 10){\textcolor{red}{$5$}}
     \put(73, 14){\textcolor{red}{$9$}}
     \put(95, 11){\textcolor{black}{$11$}}
      \end{overpic}
   \caption{The grid triangulation of Figure~\ref{f:santos}, together with the indices $i$ of the shedding sequence $\ba$ defined in the proof of Theorem~\ref{t:diam} (left)
   and the corresponding values of $\tau(a_i)$ (right).  A chain of maximal length $\tau(\ba) = 16$ is shown in red.}
   \label{f:shedding}
 \end{center}
\end{figure}

\medskip

This completes the construction of the shedding sequence $\ba = (a_1, \ldots, a_n)$ (See Figure~\ref{f:shedding}).  It is straightforward to count the number of
antichains of $\preceq_\ba$ obtained from this construction.  Stage 1 requires as many steps as it takes for the last column of the form $U(1 + 4j)$ to run out
of vertices $(x, y)$ with $y > \ell$.  Since each vertex $a_i$ of Stage 1 is contained in some tricolumn of the form $T(1 + 4j)$, this requires at
most $|T(1 + 4j)| = 3q\ell$ steps, each of which produces an antichain.  Similarly, Stage 2 requires as many steps as it takes for the last tricolumn of the form
$T(3 + 4j)$ to run out of vertices $(x, y)$ with $y > 2\ell$.  This requires at most $|T(3 + 4j)| = 3q\ell$ steps, each of which produces an antichain.  Finally,
each set $\{a_i\}$ is trivially an antichain, so taking the singleton of each vertex $a_i$ defined in Stage 3 yields at most $2p\ell$ antichains.

The set of antichains of $\preceq_\ba$ produced by these three cases clearly forms a partition of $\nobreak{V(G) = [p \times q]}$.  There are at most $2p\ell +
3q\ell + 3q\ell = \ell(2p + 6q)$ antichains in this partition.  Thus, since $\tau(\ba)$ is the length of some chain in $\preceq_\ba$, we have \[\tau(G) \leq
\tau(\ba) \leq \ell(2p + 6q) \leq 6\ell(p + q).\] \end{proof}

Theorems~\ref{t:size} and~\ref{t:diam} now immediately imply the following general result.

\begin{thm}
Let $G$ be a grid triangulation of $[p \times  q]$ such that every triangle
fits in an $\ell \times \ell$ subgrid.  Then $G$ can be realized as the graph
of a convex polyhedron embedded in an integer grid of size \ts
$4(pq)^3 \times 8(pq)^5 \times (500(pq)^8)^{6\ell(p+q)}$.
\label{t:grid}
\end{thm}

Corollary~\ref{c:grid} now follows by setting $p=q=k$.

\bigskip

\section{Final remarks and open problems}\label{s:fin}

\subsection{}
The study of the Quantitative Steinitz Problem was initiated by Onn and
Sturmfels in~\cite{OS}, who gave the first nontrivial upper bound on the
grid size.  For plane triangulations, a different approach was given
in~\cite{DG}.   Since then, there have been a series of
improvements (see~\cite{BS,R,Ro}), leading to the currently best $\exp O(n)$ bound
in~\cite{RRS}.  The only other class of graphs for which there is a
subexponential bound, is the class of triangulations corresponding to stacked
polytopes~\cite{DS}, which can be embedded into a polynomial size grid.

In the opposite direction, there are no non-trivial lower bounds on the
size of the grid.  If anything, all the evidence suggests that the answer
may be either polynomial or near-polynomial.  Note, for example,
that while the number of plane triangulations on~$n$ vertices is $\exp O(n)$
(see e.g.~\cite{DRS}), the number of grid polytopes in a polynomial size cube
$O(n^d)\times O(n^d) \times O(n^d)$, is superexponential, see~\cite{BV}.
Of course, many of these have isomorphic graphs.  In any event, we conjecture that
for triangulations a polynomial size grid is sufficient indeed.

\subsection{}
Our Theorem~\ref{t:grid} is a variation on results in~\cite{BR,FPP} and
can be viewed as a stand alone result in \emph{Graph Drawing}.
It is likely that the polynomial bounds in the theorem can be substantially
improved.  We refer to~\cite{TDET} for  general background in the field.

\subsection{}
Let us mention that not every grid triangulation is \emph{regular} (see~\cite{DRS}
for definitions and further references).  An example found by Santos
(quoted in~\cite{KZ}), is shown in Figure~\ref{f:santos} in the introduction.
This means that one cannot embed this triangulation by a direct lifting;
another plane embedding of the triangulation is necessary for that.

\subsection{}\label{ss:fin-diam}
The shedding diameter of a plane triangulation~$G$ is closely related and
bounded from above (up to an additive constant), by the \emph{optimal height}
of the \emph{visibility representation} of~$G$.  This is a parameter of
general graphs, defined independently in~\cite{RT,TT}, and explored
extensively in a series of recent papers by He, Zhang and others
(see e.g.~\cite{HZ,HWZ,ZH1,ZH2}).  Motivated by VLSI applications,
the results in these papers give linear upper bounds on the optimal
height of various classes, which are too weak for the desired
subexponential upper bounds in the Quantitative Steinitz's Problem.
In fact, one can view our Theorem~\ref{t:grid} as a rare sublinear
bound on the height representation of a class of graphs.

\subsection{}\label{ss:fin-rand}
While the shedding diameter is linear in the worst case, it is sublinear
in a number of special cases.  For example, for \emph{random} stacked
triangulations the shedding diameter becomes the height of a random
ternary tree, or $\theta(\sqrt{n})$, see e.g.~\cite{FS}.  For the
(nearly-) balanced stacked triangulations~$G$ we have $\tau(G)=O(\log n)$,
giving a nearly polynomial upper bound in the Quantitative Steinitz's Problem.
While these cases are covered by a
polynomial bound in~\cite{DS}, notice that our proof is robust enough
to generalize to other related iterative families. In fact,
we conjecture that $\tau(G) = O(\sqrt{n})$ w.h.p.,
for random triangulations with $n$ vertices (cf.~\cite{CFGN}).

\bigskip

\noindent
\textbf{Acknowledgements}  \, The authors are grateful to
Jes\'{u}s De Loera,  Stefan Felsner,  Alexander Gaifullin,
J\'{a}nos Pach, Rom Pinchasi, Carsten Thomassen, Jed Yang, and
G\"{u}nter Ziegler for helpful comments and interesting conversations.
We are especially thankful to G\"{u}nter Rote for the careful reading
of the previous draft of the manuscript, a number of useful remarks
and help with the references.  A preliminary version of this work
has appeared in the second author's Ph.D.~thesis~\cite{Wi}.
The first author was partially
supported by the BSF and NSF.

\end{document}